\documentclass[runningheads]{llncs}

\usepackage{amsfonts}
\usepackage{mathtools}
\usepackage{graphicx}
\usepackage{color}
\usepackage{csquotes}
\usepackage{amssymb}
\usepackage{amsmath}
\usepackage{stmaryrd}
\usepackage{enumerate}
\usepackage{bussproofs}
\usepackage[all]{xy}

\usepackage{todonotes}

\usepackage{xcolor}

\definecolor{darkolivegreen}{rgb}{0.33, 0.42, 0.18}
\definecolor{darkpurple}{rgb}{0.4, 0.0, 0.4}

\newcommand{\vds}{\vdash_{\mathsf{S}}}

\newcommand{\sab}{S_{a,b}}

\begin{document}

\title{A logical perspective on intending to keep a true secret
}
\titlerunning{A logical perspective on intending to keep a true secret}
%
\author{Alessandro Aldini\inst{1}\orcidID{0000-0002-7250-5011} \and Davide Fazio\inst{2}\orcidID{0000-0001-8136-732X} \and
Pierluigi Graziani\inst{3}\orcidID{0000-0002-8828-8920}\and
Raffaele Mascella\inst{4}\orcidID{0000-0002-1305-7853} \and Mirko Tagliaferri\inst{5}\orcidID{0000-0003-3875-0512}}
\authorrunning{Aldini, Fazio, Graziani, Mascella and Tagliaferri}
%
\institute{University of Urbino, Italy,
\email{alessandro.aldini@uniurb.it}\and 
University of Teramo, Italy, 
\email{dfazio2@unite.it}\and
University of Urbino, Italy,
\email{pierluigi.graziani@uniurb.it}\\
\and
University of Teramo, Italy, \email{rmascella@unite.it}\and
University of Urbino, Italy,
\email{mirko.tagliaferri@uniurb.it}}

\maketitle              
\begin{abstract}
Logical investigations of the notion of secrecy are typically concentrated on tools for deducing whether private information is well hidden from unauthorized, direct, or indirect access attempts.
This paper proposes a multi-agent, normal multi-modal logic to capture salient features of secrecy's intentions. Specifically, we focus on the intentions, beliefs, and knowledge of secret keepers and, more generally, of all the actors involved in secret-keeping scenarios. In particular, we investigate intentions underlying the keeping of a true secret, namely a secret concerning information known (and so true) by the secret keeper. 
The resulting characterization of \textit{intending to keep a true secret} provides valuable insights into conditions ensuring or undermining secrecy depending on agents' attitudes and links between secrets and their surrounding context.
We present the proposed logical system's soundness, completeness, and decidability results. Furthermore, we outline some theorems with potential applications to several fields, e.g., computer science and the social sciences.

\end{abstract}

\begin{flushright}
{\ttfamily} \textit{Is your man secret?\\
Did you ne'er hear say,\\
'Two may keep counsel, putting one away'?}\\ 
by William Shakespeare,\\
Romeo and Juliet, Act II SCENE 4.
\end{flushright}

\section{Introduction}\label{introd}

Secrets are part of our private and public lives and concern information we wish to remain confidential. Keeping a cooking recipe secret from competitors or keeping home banking credentials secret from hackers are just a few examples. However, what characterizes secrets is the separation between agents who hold the information and agents from whom that information is hidden.

Studies on the notion of secrecy abound in specific literature. There are, for example, publications on secrets in the field of psychology \cite{Slepian2017,Slepian2022,Kelly2002}, philosophy \cite{Bellman1981}, computer science \cite{Halpern2008}, and semiotics \cite{Volli2020}. Alongside these studies, there is also research on secrecy from the perspective of formal logic \cite{Ismail2020,Xiong2023}. 
Our contribution lies in this stream of research.
Indeed, logical investigations have made formally explicit salient aspects for understanding the concept of a secret, such as the distinction between \textit{knowing a secret} (\cite{Xiong2023}) and \textit{keeping/intending to keep a secret}\footnote{See Slepian \cite{Slepian2017} for an accurate analysis of this distinction in psychology.}. Nevertheless, much more work needs to be done. In particular, the role of \textit{intentionality} in keeping a secret requires a more in-depth investigation.

The present article pushes forward logical investigations concerning the concept of ``knowing a (true) secret'' carried out in \cite{Xiong2023} with the aim of extending them to cope with the notion of \emph{intending to keep a true secret}  by highlighting its epistemic and intentional content. In particular, we are interested in the formal treatment of propositions like ``Agent $a$ \emph{intends} to keep $\varphi$ \emph{secret} from agent $b$" and related notions. In this respect, our key idea is close to, e.g., \cite{Slepian2022a,Slepian2022}, where, instead of being conceived as the action of concealment of pieces of information, secrecy is more generally defined as ``an intention to keep some piece of information, known to oneself, unknown from one or more others'' (cf. \cite[p. 542]{Slepian2022a}). In addition, we take a different approach with respect to \cite{Xiong2023} by including within our definition of keeping a true secret two distinct epistemic notions (knowledge and beliefs). This is done to emphasize that agents may not have privileged access to the mental states of other agents. In turn, this allows us to describe scenarios in which an agent has the intention of keeping a secret even though the information within the secret is known by other agents in the system.\\ Our investigation will be carried out from a static and descriptive perspective. Indeed, some dynamic aspects of secrecy are intentionally left aside, e.g., belief revision/updating of agents' knowledge due to the revelation of secrets or possible interactions between secret keepers. As a final output, we outline a system that can be regarded as  ``intermediate'' between the ones outlined in \cite{Xiong2023} and \cite{Ismail2020} (cf. Section \ref{rel}), and as a new tool for investigating intentions of secrecy, tying those to important epistemic notions such as, e.g., beliefs.
 
The paper is structured as follows. Section \ref{sect:motiv} discusses and motivates some desired requirements of the notion of secret and introduces our definition informally. Based on the provided intuitions, Section \ref{sec:a logical system} presents a logical system for reasoning about keeping a true secret and proves the meta-theorems regarding our system's consistency, completeness, and decidability.
Section \ref{sec some properties of true} shows some fundamental properties about true secrets, while Section \ref{sec:relation between} investigates interesting relations between our logical framework and the theory of information flow analysis. Section \ref{rel} briefly compares our proposal and other systems available in the literature. Finally, in Section \ref{concl}, we discuss future works.

\section{Formalizing true secrets}\label{sect:motiv}

In \cite{Xiong2023}, Z. Xiong and T. {\AA}gotnes investigate the concept of ``knowing a secret'' as a modality. Indeed, they combine standard notions of knowledge and ignorance from epistemic logic to obtain a formal account of situations in which an agent knows a proposition and knows that everyone else does not know it. Upon considering a non-empty set $Ag$ of agents, they define the proposition ``$a$ knows the secret $\varphi$'' ($S_{a}\varphi$) as \[S_{a}\varphi:=K_{a}(\varphi\land (\bigwedge_{b\in Ag\smallsetminus\{a\}}\neg K_{b}\varphi)).\]
Actually, the definition provided by Xiong-{\AA}gotnes is rather meaningful. Indeed, it captures two important epistemic features of secrecy's intentions, namely the fact that secret keepers actually \emph{know} the proposition intended to be kept secret and \emph{know} that anyone else does not know it. However, as the same authors point out, the important thing ``is not that $b$ actually should not know, but that [...] \emph{$a$ should know (or believe) that $b$ doesn’t know}. This property of secrets is [...] often implicitly assumed,
e.g., in formulations such as “… \emph{intended to be kept hidden}”.

For the sake of completeness, and in order to make clear similarities and differences between the Xiong-{\AA}gotnes operator and our, we summarize some results they obtain in Table~\ref{tab:agot} (see \cite[Proposition 1--26]{Xiong2023} for details). Following notational conventions from \cite{Xiong2023}, for any $\varphi$ in the language of classical epistemic logic, we write $\models_{\star}\varphi$, where $\star$ is T or S5, to denote that $\varphi$ is provable in the usual normal epistemic logic T or S5, respectively, and we write simply $\models\varphi$ to denote that $\varphi$ is provable in the basic system K.

\begin{table}[t]
   \centering
        \setlength{\tabcolsep}{5pt}
    \begin{tabular}{|p{7cm}|p{8cm}|}
        \hline
        1. $\models_{T} S_{a}\varphi \to (K_{a}\varphi \land \neg K_{b}\varphi)$, when $a\ne b$. & 
        15. $\models_{S5}S_{a}\varphi\to S_{a}S_{a}\varphi$. \\
        
        2. $\models_{T}\neg K_{b}S_{a}\varphi$, when $a\ne b$. & 
        16. $\not\models_{S5}\neg S_{a}\neg S_{a}\varphi$. \\
        
        3. $\models_{S5}\neg S_{a}K_{b}\varphi$, when $a\ne b$. & 
        17. $\not\models_{S5}\neg S_{a}\varphi\to S_{a}\neg S_{a}\varphi$ and $\not\models_{S5}\neg S_{a}\varphi\to \neg S_{a}\neg S_{a}\varphi$. \\
        
        4. $\models_{S5}\neg S_{a}\neg K_{b}\varphi$, when $a\ne b$. & 
        18. $\models S_{a}(\varphi\to\psi)\to (S_{a}\varphi\to S_{a}\psi)$. \\
        
        5. $\models_{S5}K_{a}S_{a}\varphi \lor K_{a}\neg S_{a}\varphi$. & 
        19. $\models(S_{a}\varphi \land S_{a}\psi)\to S_{a}(\varphi\land\psi)$. \\
        
        6. $\models_{S5}\neg S_{a}\varphi \to  K_{a}\neg S_{a}\varphi$. & 
        20. $\not\models_{S5}S_{a}(\varphi \land \psi)\to S_{a}\varphi$. \\
        
        7. $\models_{S5}S_{a}\varphi \to K_{a}S_{a}\varphi$. & 
        21. For any $\varphi$, $\not\models_{S5}S_{a}\varphi$. \\
        
        8. $\models_{T} S_{a}\varphi \to \neg S_{b}\varphi$, when $a\ne b$. & 
        22. If $\models \varphi$, then $\models \neg S_{a}\varphi$. If $\models\varphi$ then $\models\neg S_{a}\neg\varphi$. \\
        
        9. $\models_{S5}\neg S_{a}S_{b}\varphi$, when $a\ne b$. & 
        23. $\not\models_{S5}\neg S_{a}(\varphi\to\psi)\to(\neg S_{a}\varphi\to\neg S_{a}\psi)$. \\
        
        10. $\models_{S5}\neg S_{a}\neg S_{b}\varphi$, when $a\ne b$. & 
        24. For some $\varphi,\psi$,  $\models\varphi\to\psi$, but  $\not\models_{S5}S_{a}\varphi\to S_{a}\psi$. \\
        
        11. $\models_{T}S_{a}\varphi\to\neg S_{a}\neg\varphi$ but $\not\models_{S5}\neg S_{a}\neg\varphi\to S_{a}\varphi$. & 
        25. For some $\varphi,\psi$,  $\models\varphi\to\psi$, but $\not\models_{S5}\neg S_{a}\varphi\to \neg S_{a}\psi$. \\
        
        12. $\models_{T}S_{a}\varphi\to\varphi$. & 
        26. For some $\varphi,\psi$,  $\models\varphi\to\psi$, but $\not\models_{S5}S_{a}\neg\varphi\to S_{a}\neg\psi$. \\
        
        13. $\models_{T}\neg S_{a}\bot$. & 
        27. For some $\varphi,\psi$,  $\models\varphi\to\psi$, but $\not\models_{S5}\neg S_{a}\neg\varphi\to \neg S_{a}\neg\psi$. \\
        
        14. $\models_{T}\neg S_{a}\top$. & 
        28. If $\models \varphi\leftrightarrow\psi$ then $\models S_{a}\varphi\leftrightarrow S_{a}\psi$. \\
        \hline
    \end{tabular}
    \caption{Some results from Xiong-{\AA}gotnes' \cite{Xiong2023}}
    \label{tab:agot}
\end{table}
As the reader may notice, $S_{a}$ behaves adequately in formalizing epistemic aspects of secrecy since it allows one to describe epistemic differences between those who know a secret and outsiders, it grants that secret keepers are always aware of secrets they conceal (they know the secret, they know of being secret keepers or not), and enlighten an interesting fact: no proposition that is known by anyone or, at least by outsiders, is a known secret. Among other things, secrets are always ``reflexive''; namely, having a secret is (at least whenever the S5 framework is assumed) always equivalent to having the secret of having a secret.

A further, interesting property of $S_{a}$ is that it satisfies what Xiong-{\AA}gotnes call the \emph{interpolation rule}, namely for any $\varphi,\psi,\chi$:
\[\text{if} \models\varphi\to\psi\text{ and }\models\psi\to\chi,\text{ then }\models S_{a}\varphi\land S_{a}\chi\to S_{a}\psi.\label{eq:intprule}\]
This inference rule ensures that secret knowledge is closed under possibly long chains of deductions (cf. below).

It is worth noticing that, among peculiar properties of $S_{a}$, and of ``having a secret'' we find what \cite{Xiong2023} calls \emph{Secret negation completeness} (item 5 in Table \ref{tab:agot}). Indeed, at least whenever the S5 framework is taken into account, any agent always knows whether she has or has not a secret. As we will see, this is a rather strong property that does not hold, whenever the intention of keeping a secret is considered as it presupposes not only negative introspection w.r.t. knowledge/belief, but 
 also some kind of negative introspection w.r.t. intentions (see Proposition \ref{prop:notnegcompl} below).

As \cite{Xiong2023} highlights, $S_{a}$ does not take into account \emph{intentional} aspects of secrecy. This sets their work as a genuine attempt to formalize the notion of ``having a secret'' which does not involve \emph{per sé} any intentional component (cf.  \cite{Slepian2022a,Slepian2017}). However, as already mentioned above, whenever the concept of ``keeping a secret'' is considered, one has to take into account its intentional aspects. Indeed, \cite{Slepian2022a} provides a new approach to secrecy dynamics, as it distinguishes within them, against previous approaches describing the concept of keeping a secret as an action of concealment, two different components. The first is an intention, namely $(a)$ the intention of keeping something secret. It does not require actions to be performed but is rather a predisposition to conceal a piece of information from one or more individuals \emph{provided that the context of reference requires it}. The second is $(b)$ an action of concealment that takes place provided that definite pre-conditions, or contexts, require it in order to avoid undesired information leakages. 
Our paper will be devoted to providing a formal treatment of $(a)$, while we refer the reader, e.g., to \cite{Hoek2012} for a logical investigation of concealment and revelation actions.

The specific notion of ``intending to keep a secret'' we will formalize relies on a scenario with an agent, call it $a$, representing a secret keeper, and (at least\footnote{Think of a group of 
individuals, e.g. a company, taken as a whole.}) another agent, call it $b$, representing the 
nescient, to whom it is agent $a$'s intention that the secret not be revealed.

Following \cite{Slepian2022a}, the intentionality component of ``intending to keep a secret'' we are after to formalize should not be meant as an intention of acting effectively to reach a certain goal, but rather as an intention to bring about a state of affairs with certain features that an agent $a$ has \emph{as long as the agent's reference context sets it as preferable for the achievement of} $a$'s goals. Given this understanding of ``intending to'', we will focus on agent $a$'s intention of preserving both the truth and the secrecy of a 
certain secret, as it happens, for instance, in the message exchange of cryptographic protocols 
over the Internet or in the case of database management systems protecting sensitive 
repositories. Hence, for our definition, we are not interested in scenarios in which agent $a$ 
might intend to alter the object of the secret.

Our notion of ``intending to keep a true secret'' is based on the following intuitive assumptions that, extending the notion of ``having a secret'' already treated by \cite{Xiong2023}, concentrate on the standpoint of agent $a$:
\begin{enumerate}
    \item Agent $a$ knows the object of the secret (called \textit{secretum}); since we also assume that knowledge is factive, it follows that the content of $a$'s secret is true. Hence the expression \textit{true secret}.
    \item Agent $a$ believes that agent $b$ does not know the \textit{secretum};
    we use beliefs rather than knowledge because we assume that agent $a$ has no direct access to $b$'s mental attitudes.
    \item Agent $a$ intends to behave in such a way that agent $b$ does not end up knowing the \textit{secretum}, while preserving its truth. More precisely, it is $a$'s intention to preserve the integrity of the \textit{secretum}. Both intentions are held provided that the agent's reference context sets their objects as preferable, or suitable, for the achievement of $a$'s goals.\footnote{If not specified otherwise, we will always treat ``intending to'' in this precise sense.}
\end{enumerate}
To formalize such assumptions, we employ three modal, multi-agent operators expressing knowledge ($K$), belief ($B$), and intentionality ($I$), combined\footnote{{The interaction between knowledge and belief in formal epistemology has been a subject of considerable study for several decades, with early foundational work laid by Hintikka \cite{Hintikka1962}. For further literature on the connection between $K$ and $B$ see \cite{Rasmus2024}. An important line of research in epistemic logic also explores the concept of intentionality. For further reading on this topic, see \cite{Lorini2008}.}} to obtain the following notion of true secret:
\[
S_{a,b}\varphi := K_{a}\varphi \wedge B_{a}\neg K_{b}\varphi \wedge I_{a} ( \varphi \land \neg K_{b} \varphi)
\label{sstar}
\]
where formula $\varphi$ represents the \textit{secretum}.

The first two conjuncts express items $1.$ and $2.$, respectively. The epistemic operator $K_i$ for knowledge and the doxastic operator $B_i$ for belief are defined for each agent $i$. Hence, the formulas should be read as: ``agent $a$ knows $\varphi$'' ($K_a \varphi$) and ``agent $a$ believes that agent $b$ does not know $\varphi$'' ($B_a \neg K_b \varphi$). 
In particular, we recall the assumption we take as the main distinguishing feature between knowledge and belief, i.e., knowledge is factive ($K_a \varphi \rightarrow \varphi$) while belief is not. Indeed, we stress that we are imposing neither that agent $a$ knows that agent $b$ does not know $\varphi$, nor that agent $b$ indeed does not know $\varphi$. Instead, $B_{a}\neg K_{b}\varphi$ encodes the point of view of $a$ w.r.t. $b$'s knowledge about $\varphi$\footnote{An analogous relaxation is common practice, e.g., in the setting of cryptography, where the theoretical but unrealistic assumption of  \textit{perfect security} is replaced, for all practical purposes, by the notion of  \textit{semantic security}. Intuitively, the former assumes no information is revealed, whereas the latter implies that some information can be feasibly extracted (however, such an event shall be negligible).}.

The third conjunct expresses item $3.$ and represents a concealing component formalizing the intention of agent $a$ to sustain $b$'s ignorance about $\varphi$.  
The specific condition of ignorance for agent $b$ is called \textit{factive ignorance}, which goes back to Fitch~\cite{Fitch1963,Goranko2021,CIFMA20,Aldini2023,Tagliaferri2023} and is defined as $T_{b} \varphi:=\varphi \land \neg K_b \varphi$ (this motivates the reference to the concept of \textit{true secret}).

The intentionality operator $I_a \varphi$ is inspired (with certain interpretational modifications) by Colombetti~\cite{Colombetti1999} and can be read as ``agent $a$ intends to bring about a state of affairs in which $\varphi$ is true as long as in $a$'s context of reference $\varphi$ is preferable for the achievement of $a$'s goals''. Indeed, using an intentionality operator $I_a$ rather than a communication operator is a crucial aspect of our approach\footnote{
A fundamental difference between
our $I_{a}$ intention mode and, e.g., communication-based updates in dynamic epistemic logic (DEL) lies in the fact that intention is not inherently dynamic in the sense of DEL but rather goal-oriented.}.
In fact, by replacing the third conjunct with a statement meaning ``agent $a$ does not communicate $\varphi$ to agent $b$'', we would be talking about an omission rather than a secret. Keeping information secret is a disposition to keep information confidential. It may be motivated by various factors, such as privacy, data protection, or maintaining a competitive advantage. Differently, omitting information is tantamount to not actively communicating information. This does not necessarily imply an intention to keep the information secret; it may depend on a mistake, ignorance, or the assumption that agent $b$ already knows the information. Furthermore, keeping a secret might not involve just the will of not communicating its content, but it may entail also the intention of acting directly or indirectly to avoid that the nescient might obtain information through other direct or indirect sources.
Moreover, not only agent $a$ is considered to know her intentions (\textit{transparency} condition), but her intention to bring about a state of affairs in which $\varphi$ is true is equivalent to intending to bring about a state of affairs in which she \emph{knows} that  $\varphi$ is true (\textit{awareness} condition).\label{transparawar}

As we will state formally, the formula $I_a T_b \varphi$ subsumes $I_a \varphi$, $I_a K_a \varphi$, and $I_a \neg K_b \varphi$, which are all jointly necessary to capture the intentions that agent $a$ must manifest to properly ensure our idea of ``intending to keep a true secret''.
$I_a \neg K_b \varphi$ expresses that agent $a$ does not want agent $b$ to know the \textit{secretum}.  $I_a \varphi$ and $I_a K_a \varphi$ express the intentions of preserving truth and agent $a$'s knowledge of $\varphi$, respectively. Of course, as already recalled above, all intentions must be understood as holding as long as their objects align with the agent's goals. Actually, assuming these conditions is necessary to exclude undesirable strategies. For instance, by omitting $I_a \varphi$, agent $a$ may be open to the intention of falsifying $\varphi$, which obviously would keep agent $b$ ignorant about $\varphi$ (disregarding that preserving the truth of $\varphi$ is as important as guaranteeing its confidentiality for agent $b$). Similarly, by omitting $I_a K_a \varphi$, agent $a$ could simply intend to conceal $\varphi$ to everyone, including $a$ herself. These ideas are counterintuitive for our purposes, as shown by the following practical example. Imagine we had a machine assigned to protect the accesses to a Bitcoin digital wallet. We would not want that, to preserve the confidentiality of the proposition ``The string \texttt{E9873D***} is the private key of wallet $W$'', the machine either (i) destroys the wallet (thus making the proposition false) or (ii) erases and forgets the unique copy of the key\footnote{Around 20\% of the existing Bitcoin appear to be in lost or otherwise stranded wallets, the confidentiality of which is not in question~\cite{NYT}.} (thus making the proposition, which would remain true, unknown to the machine).

{Our approach emphasizes the complex interactions among knowledge, belief, and intention within a multimodal context, illuminating the dependencies and relationships between these modalities. Additionally, it would be worthwhile to consider an approach that treats secrecy as a primitive modality, as explored in \cite{Xiong2023}. This second perspective can potentially streamline the semantics and axioms of the modal system, leading to valuable insights into the properties of secrecy in epistemic contexts. Both approaches offer distinct advantages, yet they cater to different analytical objectives. While we acknowledge the benefits of treating secrecy as a primitive modality, our current research focuses on the interdependencies among knowledge, intention, and belief, demonstrating how these modalities contribute to a more nuanced understanding of secrecy within a broader epistemic framework. We intend to reserve the further exploration of modeling secrecy as a primitive modality for future research.}

In general, the conditions that characterize our notion of secret depend on the properties we assume for the three underlying modal operators (e.g., factivity of $K_i$ and awareness of $I_i$). Such an analysis will be the subject of the following sections, where we will present the formal syntax and semantics of a modal logic encompassing all the presented ingredients.

\section{A logical system for true secrets}\label{sec:a logical system}

This section introduces the framework in which our notion of ``intending to keep a secret'', or ``keeping a secret" in a nutshell\footnote{In the rest of this article, we will always make use of the abbreviation ``keeping a secret" whenever it does not create misunderstandings. Therefore, in most cases, although it must always be considered an essential part of our definition, the intentional component of ``intending to keep a true secret'' will be left implicit.} will be treated formally. As it will be clear, the machinery introduced is a normal modal logic called $\mathsf{S}$, endowed with intention, knowledge, and belief operators satisfying minimal conditions.

Let $Ag$ be a non-empty, finite set of agents, and let \textit{Var} be an infinite, countable set of variables. Let $\mathrm{Fm}_{\mathsf{S}}$ be the smallest set of formulas generated by the following grammar:
$$\varphi:=\ p\ |\ \neg\varphi\ |\ \varphi\land\varphi\ |\ I_{a}\varphi\ |\ K_{a}\varphi\ |\ B_{a}\varphi$$
where $p\in \mathit{Var}$ and $a\in Ag$. As customary, we set $\varphi\lor\psi:=\neg(\neg\varphi\land\neg\psi)$ and $\varphi\to\psi:=(\neg\varphi\lor\psi)$. \\
The logic $\mathsf{S}=\langle\mathrm{Fm}_{\mathsf{S}},\vdash_{\mathsf{S}}\rangle$ is the \emph{derivability} relation (to be defined as customary) induced by the axiom and inference rule schemes, for any $a\in Ag$, illustrated in Table~\ref{tab:assiomi}.
\begin{table}[t]
\begin{center}
\begin{tabular}{| p{2em} l|}
\hline & \\
A1 & All tautologies of classical propositional logic \\
A2 & $\star(\varphi\to\psi)\to(\star\varphi\to\star\varphi)$ \qquad for any $\star\in\{B_{a}, K_{a},I_{a}\}$ \\
A3 & $K_{a}\varphi\to\varphi$ \\
A4 & $K_{a}\varphi\to K_{a}K_{a}\varphi$ \\
A5 & $B_{a}\varphi\to\neg B_{a}\neg\varphi$ \\
A6 & $K_{a}\varphi\to B_{a}\varphi$ \\
A7 & $B_{a}\varphi\to K_{a}B_{a}\varphi$ \\
A8 & $I_{a}\varphi\to\neg I_{a}\neg\varphi$ \\
A9 & $I_{a}\varphi\to K_{a}I_{a}\varphi$ \\ 
A10 & $I_{a}\varphi\to I_{a}K_{a}\varphi$ \\ 
A11 & $I_{a}\varphi\to I_{a}I_{a}\varphi$ \\ & \\
\multicolumn{2}{|c|}{
    \hspace*{4mm}
    \AxiomC{$\varphi$}
    \AxiomC{$\varphi\to\psi$}\RightLabel{RMP}
    \BinaryInfC{$\psi$}
    \DisplayProof \qquad and \qquad \AxiomC{$\varphi$}\RightLabel{$\mathrm{RN}_{\star}$}
    \UnaryInfC{$\star\varphi$}
    \DisplayProof \qquad
    where $\star\in\{I_{a},K_{a},B_{a}\}$ \hspace*{4mm}
} 
\\ & \\
\hline
\end{tabular}
\end{center}
\caption{Axioms and rules of $\mathsf{S}$.}
\label{tab:assiomi}
\end{table}
The epistemic part of the axiomatization contains nothing new and is, therefore, standard. We confine ourselves to point out that, in order to keep our framework \emph{as minimal as possible}, we will not make use of negative introspection as \cite{Xiong2023} does for obtaining some of their results. Focusing on the intentionality operator $I$,
axiom A8 states that any agent $a$ is consistent with her intentions, while 
axioms A9 and A10 express the transparency and awareness conditions mentioned in the previous Section. As regards axiom A11, it might be regarded as a \emph{rationality} condition for intentionality. Having the intention to bring about a state of affairs in which $\varphi$ is true as long as the truth of $\varphi$ is, given the context of reference, preferable for the achievement of one's goals entails that such an intention is preserved in all states of affairs reachable through it. We remark that, having the intention to bring about a state of affairs in which, e.g., ``the door is closed'' is true does not mean that such an intention reflects an effective will of preserving or determining the truth of ``the door is closed'' indefinitely. It means that any agent $a$ who intends to bring about a state of affairs in which ``the door is closed'' is true as long as it is preferable w.r.t. her objectives has always also the intention to bring about a state of affairs in which $a$'s intention is preserved as long as it aligns with her goals. In other words, no agent having an intention that aligns with her goals may not have the intention to bring about a state of affairs in which her intention is preserved as long as it aligns with her goals.
Therefore, the operator $I$ should not be regarded as a \emph{tense} operator. Also, it is worth observing that $I$, like $B$, is not factive and, differently from $B$, is not implied by $K$\footnote{Note that there is no axiom that makes the intentional operator $I_a$ factive as is the case for $K_a$ (i.e., there is no axiom of the form $I_a\varphi\to\varphi$). This implies that intention does not collapse onto an intention of \emph{truth preservation} but it is rather to be meant as an intention of acting to bring about a state of affairs in which $\varphi$ is \emph{still} true, if it already does, and $\varphi$ \emph{becomes} true if it is not yet.}.

The axioms of $\mathsf{S}$ suggest that knowledge $K_a$ and intention $I_a$ are independent in their base functionality. Indeed, we assume that agents need not intend to bring about a state of affairs in which a given proposition $\varphi$ holds, if they know it does. This is motivated by counterexamples occurring in everyday life. If I know that the door is open, then this does not mean that I have the intention to bring about a state of affairs (e.g. by acting in some way) in which the door is still open. The motivation for not assuming the converse is clear. These axioms tell us that intending $\varphi$ implies knowing that the intention holds ($K_a I_a \varphi$), and it also implies that the intention includes knowledge ($I_a K_a \varphi$). Therefore, the system formalizes that if an agent intends $\varphi$, they are also aware of their own intention, and they intend to know $\varphi$ as well. This creates a relationship where intention presupposes a certain level of self-awareness and knowledge. Axioms A9 and A10 have the role of determining a rather strong notion of “intending to”. Indeed, we assume that agents are always aware (where aware is naively expressed as “knowing that”) of their intention. This is the content of A9. Moreover, we assume that intentions are “intentions of success” in the following sense. If an agent aims at bringing about a state of affairs in which $\varphi$ is true, then she intends to bring about a state of affairs in which she knows her intention has been successfully fulfilled. We consider “aware intention”, “intention to be successful in obtaining what it is intended to”, and “intention” as one and the same thing. Moreover, we observe that these assumptions bring with them conditions that might look quite strong at first sight. For example, it can be proven that $\vdash_{\mathsf{S}}\varphi\to\psi$ implies $\vdash_{\mathsf{S}}I_{a}\varphi\to I_{a}K_{a}\psi$. In other words, our agents are not only aware of their intentions, and they intend to be successful, but they are fully aware (in the above senses) of the material consequences of their intentions.  However, although their combination with axioms for normal modalities makes them bearer of (at first sight) quite strong consequences, the above assumptions seem reasonable when formalizing the kind of intention of concealment playing a key role in secrecy scenarios.

Let $A$ be a non-empty set. Recall that a binary relation $R\subseteq A\times A$ is said to be \emph{serial} provided that, for any $i\in A$, there exists $j\in A$ such that $R(i,j)$.

\begin{definition}\label{def:s-mod}An \emph{$\mathcal{S}$-frame} (a \emph{frame}, for short) is a tuple $$\mathcal{F}=(W,\{R^{I}_{a}\}_{a\in Ag},\{R^{K}_{a}\}_{a\in Ag},\{R^{B}_{a}\}_{a\in Ag})$$ such that:
\begin{enumerate}
    \item $W$ is a non-empty set of worlds (or states);
    \item $R^{B}_{a}\subseteq W\times W$ is serial;
    \item For any $a\in Ag$, $R^{I}_{a}\subseteq W\times W$ is serial and transitive;
    \item For any $a\in Ag$, $R^{K}_{a}\subseteq W\times W$ is reflexive and transitive;
    \item For any $i,j,w\in W$, and any $a\in Ag$, if $R^{K}_{a}(i,j)$ and $R^{I}_{a}(j,w)$, then $R^{I}_{a}(i,w)$.
    \item For any $a\in Ag$, $R^{B}_{a}\subseteq R^{K}_{a}$;
    \item For any $i,j,w\in W$ and $a\in Ag$, if $R^{K}_{a}(i,j)$ and $R^{B}_{a}(j,w)$, then $R^{B}_{a}(i,w)$;
    \item For any $i,j,w\in W$ and $a\in Ag$, $R^{I}_{a}(i,j)$ and $R^{K}_{a}(j,w)$ imply $R^{I}_{a}(i,w)$.
\end{enumerate}
\end{definition}
 
\begin{definition}
    An $\mathcal{S}$-model (a \emph{model}, in brief) is a tuple $$\mathcal{M}=(W,\{R^{I}_{a}\}_{a\in Ag},\{R^{K}_{a}\}_{a\in Ag},\{R^{B}_{a}\}_{a\in Ag},v)$$ such that
    \begin{enumerate}
        \item $(W,\{R^{I}_{a}\}_{a\in Ag},\{R^{K}_{a}\}_{a\in Ag},\{R^{B}_{a}\}_{a\in Ag})$ is an $\mathcal{S}$-frame, and
        \item $v: \mathit{Var}\to\mathcal{P}(W)$ is a mapping, called an \emph{evaluation}.
    \end{enumerate}
\end{definition}
As customary, given a model $\mathcal{M}=(W,\{R^{I}_{a}\}_{a\in Ag},\{R^{K}_{a}\}_{a\in Ag},\{R^{B}_{a}\}_{a\in Ag},v)$, $v$ can be regarded as a mapping assigning to any $p\in \mathit{Var}$ the set $v(p)$ of worlds in which $p$ is true. 

\begin{definition}
    For any model $\mathcal{M}$ and any $i\in W$, 
    the truth of $\varphi\in \mathrm{Fm}_{\mathsf{S}}$ at $i$ w.r.t.~$\mathcal{M}$, written $\mathcal{M},i\models\varphi$, is defined recursively as follows
\begin{itemize}
    \item $\mathcal{M},i\models p$ iff $i\in v(p)$;
    \item $\mathcal{M},i\models \neg\varphi$ iff $\mathcal{M},i\not\models \varphi$;
    \item $\mathcal{M},i\models \varphi\land\psi$ iff $\mathcal{M},i\models\varphi$ and $\mathcal{M},i\models \psi$;
    \item $\mathcal{M},i\models \star_{a}\varphi$ iff, for any $j\in W$, if $R^{\star}_{a}(i,j)$, then $\mathcal{M},j\models \varphi$, with $\star\in\{K,B,I\}$ and $a\in Ag$.
\end{itemize}
\end{definition}

\label{validitydef}Let $\varphi\in \mathrm{Fm}_{\mathsf{S}}$ and $\mathcal{M}$ be a model. We say that $\varphi$ is true in $\mathcal{M}$, written $\mathcal{M}\models\varphi$ provided that $\mathcal{M},i\models\varphi$, for any $i\in W$. Also, we say that $\varphi$ is $\mathcal{S}$-valid, written $\models_{\mathcal{S}}\varphi$, if $\mathcal{M}\models\varphi$, for any model $\mathcal{M}$.


\noindent {The following result is a direct consequence of the well-known Sahlqvist's completeness theorem (see \cite[Theorem 4.42]{blackburn2002modal}) whose application is made available by observing that the system we deal with in this paper is obtained upon extending $\mathsf{K}$ with Sahlqvist axioms (\cite[Definition 3.51]{blackburn2002modal}).}

\begin{theorem}\label{compl-sound}For any $\varphi\in\mathrm{Fm}_{\mathsf{S}}$:
\[\vdash_{\mathsf{S}}\varphi\text{ iff }\models_{\mathcal{S}}\varphi.\]
\end{theorem}

As there are non-trivial $\mathcal{S}$-models (see, e.g., Proposition \ref{ex:notrivializ}), Theorem \ref{compl-sound} shows that our system is \emph{consistent}. Also, its proof yields that $\mathsf{S}$ is \emph{canonical} (cf. \cite[p. 56]{Hughes1}). Consequently, one has that $\mathsf{S}$ is \emph{compact} (see \cite[p.104]{Hughes1}).

Given the above discourse, through the next sections, $\models_{\mathcal{S}}$ and $\vdash_{\mathsf{S}}$ will be regarded as interchangeable. Moreover, it is worth noticing that our formal system does not cause undesired collapses of modalities we are dealing with in this paper. Let $A$ be a non-empty set. We set $\Delta_{A}:=\{(i,i):i\in A\}.$
\begin{proposition}\label{ex:notrivializ}Let $p\in \mathit{Var}$. Then:
\begin{enumerate}
    \item $\not\vds I_{a}p\to B_{a}p$;
    \item $\not\vds I_{a}p\to K_{a}p$;
    \item $\not\vds B_{a}p\to K_{a}p$.
\end{enumerate}
\end{proposition}
\begin{proof}Let us consider just (1) and (2), leaving the remaining item as a simple exercise for the reader. Let $$\mathcal{M}=(W=\{i,j\},\{R^{I}_{a}\}_{a\in Ag},\{R^{K}_{a}\}_{a\in Ag},\{R^{B}_{a}\}_{a\in Ag},v)$$ be such that, for any $a\in Ag$:
\begin{itemize}
    \item $R^{K}_{a}=R^{B}_{a}:=\Delta_{W}$; and
    \item $R^{I}_{a}:=\{(i,j),(j,j)\}.$
\end{itemize}
A routine check shows that $\mathcal{M}$ is an $\mathcal{S}$-frame. Moreover, let $v:\mathit{Var}\to\mathcal{P}(W)$ be such that for some fixed $p\in \mathit{Var}$, $v(p)=\{j\}$. Upon extending $v$ to an evaluation on the whole $\mathrm{Fm}_{\mathsf{S}}$, one has $\mathcal{M},i\not\models K_{a}p\lor B_{a}p$ but $\mathcal{M},i\models I_{a}p$. \qed
\end{proof}  

Of course, it naturally raises the question as to whether it is possible to prove that $\mathsf{S}$ is decidable. Indeed, a positive answer can be easily obtained by means of a quite standard filtration technique. To make this work available also to non-specialist readers, we provide a decidability proof and its  auxiliary lemmas in Appendix A.

\begin{theorem}\label{thm:dec} $\mathsf{S}$ is decidable.
\end{theorem}

{As the reader may observe, we choose to formalize the keeping of a true secret within the formal system $S4$ and $KD4$ instead of $S5$ and $KD45$. This choice, although not a standard (see section 6), can be justified precisely considering the object of the formalization, i.e. the keeping of true secrets. The decision reflects a desire to capture agents with epistemic limitations. By choosing $S4$ for knowledge, we forgo negative introspection, acknowledging that agents may not always know whether they lack knowledge, which is a more realistic portrayal of human cognition. Similarly, using $KD4$ for belief excludes negative introspection, allowing for agents who are uncertain about their own beliefs. This approach aligns with scenarios where agents are not ideal reasoners, thus providing a more flexible and nuanced representation of both knowledge and belief, without violating formal consistency.
This framework becomes particularly relevant when modeling the intention of keeping a secret. A secret involves intentionally limiting knowledge or belief to certain agents, and the absence of perfect introspection helps to account for scenarios where an agent may be unaware of their knowledge gaps or belief states, thus better reflecting real-world scenarios of secrecy, where both cognitive and informational boundaries are imperfect or deliberately managed.}

\section{Some properties of true secrets}\label{sec some properties of true}

In this section, we investigate some basic properties of the operator $S_{a,b}$. As it will be clear from the sequel, many properties of the Xiong-{\AA}gotnes operator for ``having a secret'' still hold in our framework, regardless of the addition of the intention operator $I$. In the sequel, we exhibit only the proofs of statements in which properties of the intentionality operator play a key role or which are not explicitly mentioned in \cite{Xiong2023} but we believe are somehow bearer of conceptual insights,  leaving the remaining ones to the reader.

We start by discussing some basic properties bridging our idea behind the intention of secrecy and the notion of knowing a secret of~\cite{Xiong2023}. Note that, due to results stated below, our operator satisfies the closure under equivalence rule: $\text{if} \vds\varphi\leftrightarrow\psi, \text{then} \vds S_{a,b}\varphi\leftrightarrow S_{a,b}\psi$ (see item 28 from Table \ref{tab:agot}),  the impossibility of keeping validities secret: if $\vds \varphi$, then $\vds \neg S_{a,b}\varphi$ and $\vds\neg S_{a,b}\neg\varphi$ (cf. item 21, 22 from Table \ref{tab:agot}), as well as the non-existence of tautological intentions of secrecy: $\not\vds\sab\varphi$ (for any $\varphi$).

\begin{proposition}\label{prop:propertiessab}Let $\varphi\in\mathrm{Fm}_{\mathsf{S}}$ and $a,b\in Ag$ with $a\ne b$. The following hold:
\begin{enumerate}
\item $\vds \neg S_{a,a}\varphi$;
\item $\vds \neg S_{a,b}\top$ and $\vds \neg S_{a,b}\bot$;
\item $\vds S_{a,b}\varphi\to\varphi$;
\item $\vds \sab\varphi\to K_{a}\sab\varphi$;
\item $\vds S_{a,b}\varphi\leftrightarrow S_{a,b}S_{a,b}\varphi$;
\item $\vds\neg K_{b}\varphi\to\neg K_{b}S_{a,b}\varphi$;
\item $\vds\sab\varphi\to\neg\sab\neg\varphi$;
\item $\vds\sab\varphi\to\sab K_{a}\varphi$;
\item $\vds\neg\sab S_{b, a}\varphi$.
\end{enumerate}
\end{proposition}

\begin{proof}
Throughout the proof, $\mathcal{M}=(W,\{R^{I}_{a}\}_{a\in Ag},\{R^{K}_{a}\}_{a\in Ag},\{R^{B}_{a}\}_{a\in Ag},v)$ is an  arbitrary $\mathcal{S}$-model, and $i\in W$.

\begin{enumerate}

\item[4.] By (A4), (A7) and (A9), and the distributivity of $K_{a}$ over $\land$.

\item[5.] $\vds \sab(\sab\varphi)\to\sab\varphi$ follows by item (3). To prove the converse direction, assume $\mathcal{M},i\models\sab\varphi$. By item (4), one has $\mathcal{M},i\models K_{a}\sab\varphi$. Moreover, suppose towards a contradiction that $\mathcal{M},i\not\models B_{a}\neg K_{b}\sab\varphi$. So, there is $j\in W$, such that $R^{B}_{a}(i,j)$ and $\mathcal{M},j\models K_{b}\sab\varphi$. By the definition of $\sab\varphi$ and the distributivity of $K_{b}$ over $\land$, one has $\mathcal{M},j\models K_{b}K_{a}\varphi$. Moreover, since it holds that $\vds K_b K_a \varphi \to K_b \varphi$, we have $\mathcal{M},j\models K_{b}\varphi$. But then $\mathcal{M},i\not\models B_{a}\neg K_{b}\varphi$, a contradiction. We conclude that $\mathcal{M},i\models B_{a}\neg K_{b}\sab\varphi$. Similarly, we prove that  $\mathcal{M},i \models I_{a}\neg K_{b}\sab\varphi$. Finally, $\mathcal{M},i\models I_{a}\sab\varphi$ follows by item (5). We conclude that $\mathcal{M},i\models K_{a}\sab\varphi\land B_{a}\neg K_{b}\sab\varphi\land I_{a}\neg K_{b}\sab\varphi\land I_{a}\sab\varphi$. The desired conclusion follows the distributivity of $I_{a}$ over $\land$. 
\item[8.] Assume $\mathcal{M},i\models \sab\varphi$. From $\mathcal{M},i\models K_{a}\varphi$, one has $\mathcal{M},i\models K_{a}K_{a}\varphi$, by (A4). 
Also, from $\mathcal{M},i\models K_{b}K_{a}\varphi\to K_{b}\varphi$, we obtain $\mathcal{M},i\models\neg K_{b}\varphi\to\neg K_{b}K_{a}\varphi$ and, by a customary reasoning, also $\mathcal{M},i\models B_{a}\neg K_{b}\varphi\to B_{a}\neg K_{b}K_{a}\varphi$. 
Consequently, we have $\mathcal{M},i\models B_{a}\neg K_{b}K_{a}\varphi$. $\mathcal{M},i\models I_{a}\neg K_{b}K_{a}\varphi$ can be proven similarly. Finally, note that $\mathcal{M},i\models I_{a}\varphi$ yields $\mathcal{M},i\models I_{a}K_{a}\varphi$ by (A10), and so $\mathcal{M},i\models I_{a}(\neg K_{b}K_{a}\varphi\land K_{a}\varphi)$.
\qed
\end{enumerate}
\end{proof}

We notice that all items are either proven explicitly or are easy to obtain in \cite{Xiong2023}. 
Item (1) means that an agent cannot keep a statement secret to herself.
Similarly, item (2) specifies that tautologies and contradictions cannot be kept secret from any agent. Item (2) corresponds to item 14 and item 13 from~Table \ref{tab:agot}.
Item (3) emphasizes what we meant by \textit{true secret} in Section~\ref{sect:motiv}. It corresponds to item 12 from Table \ref{tab:agot}.
Item (4) states a principle of awareness of secret: if agent $a$ keeps a proposition $\varphi$ secret from $b$, then $a$ knows to be in such a position. Hence, the operator $S$ rules out unconscious omissions of information. It corresponds to item 7 in Table~\ref{tab:agot}.
Item (5) suggests that keeping something secret is equivalent to keeping secret its secrecy. It corresponds to item 15 in Table \ref{tab:agot}. At the same time, item (6) is a direct consequence of the notion of a true secret: no agent can know the secrecy of a given proposition $\varphi$ without knowing that $\varphi$ is true. Note that we do not have in general, $\vds\neg K_{b}S_{a,b}\varphi$ (cf. item 2 from Table \ref{tab:agot}), since for us it is reasonable that agent $b$ can come to know that agent $a$ is trying to keep the information $\varphi$ secret (although, unsuccessfully in this case).
Item (7) establishes the consistency of the notion of secrecy.
Item (8) is a consequence of the factivity of knowledge, because of which 
keeping $\varphi$ secret from agent $b$ requires keeping $K_a \varphi$ secret from agent $b$. Item (7) corresponds to item 11 from Table \ref{tab:agot}, while item (8) can be derived within the Xiong-{\AA}gotnes framework, e.g., upon assuming the S5 axioms. Finally, item (9), establishes that it is not possible for an agent $a$ to keep secret from $b$ the fact that $b$ is keeping some information secret from $a$. This corresponds to item 9 from Table~\ref{tab:agot}. Interestingly, differently from \cite{Xiong2023}, for us it is possible for an agent $a$ to keep secret from an agent $b$ that $b$ is not keeping secret some information from $a$, due to the lack of negative introspection of our agents (formally: $\not\vds\neg \sab\neg S_{b,a}\varphi$, when $a\ne b$). In addition, as it will be made clear soon (cf. proposition \ref{prop:notperfsecret}) our secrets are not exclusive, since two agents could keep an information secret from each other, not realizing that they both know the information.

The following proposition states some specific properties of the intention modality in combination with the operator $S_{a,b}$.

\begin{proposition}\label{prop:propertiessab2}Let $\varphi\in\mathrm{Fm}_{\mathsf{S}}$ and $a,b\in Ag$ with $a\ne b$. The following hold:
\begin{enumerate}
\item $\vds\sab\varphi\to I_{a}\sab\varphi;$
\item $\vds\neg S_{a,b}K_{b}\varphi\land\neg\sab I_{b}\varphi\land\neg\sab B_{b}\varphi$;
\item $\vds \neg I_{a}K_{b}\sab\varphi$.
\end{enumerate}
\end{proposition}
\begin{proof}
Throughout the proof, $\mathcal{M}=(W,\{R^{I}_{a}\}_{a\in Ag},\{R^{K}_{a}\}_{a\in Ag},\{R^{B}_{a}\}_{a\in Ag},v)$ is an  arbitrary $\mathcal{S}$-model, and $i\in W$.
\begin{enumerate}

\item Suppose that $\mathcal{M},i\models\sab\varphi$. Then $\mathcal{M},i\models I_{a}(\neg K_{b}\varphi\land \varphi)$ implies that $\mathcal{M},i\models I_{a}\varphi$ by the distributivity of $I_{a}$ over $\land$. By (A10), one has that $\mathcal{M},i\models I_{a}K_{a}\varphi$. Also, since $\mathcal{M},i\models I_{a}\neg K_{b}\varphi$, one has $\mathcal{M},i\models I_{a}K_{a}\neg K_{b}\varphi$ and also $\mathcal{M},i\models I_{a}B_{a}\neg K_{b}\varphi$ by (A2) w.r.t. $I_{a}$, (A6), and (RMP). Finally, $\mathcal{M},i\models I_{a}I_{a}T_{b}\varphi$ follows from (A11).

\item Upon noticing that, if $\mathcal{M},i\models\sab K_{b}\varphi$, then $\mathcal{M},i\models K_{a}K_{b}\varphi\land B_{a}\neg K_{b}K_{b}\varphi$. By applying (A4) and a  customary reasoning, one has $\mathcal{M},i\models  K_{a} K_{b}K_{b}\varphi$ and so $\mathcal{M},i\models B_{a} K_{b}K_{b}\varphi$ by (A6). However, this is impossible by (A5). Similarly, if $\mathcal{M},i\models\sab I_{b}\varphi$, then $\mathcal{M},i\models B_{a}\neg K_{b}I_{b}\varphi$, and so $\mathcal{M},i\models B_{a}\neg I_{b}\varphi$, by (A9), $\vds (I_{b}\varphi\to K_{b}I_{b}\varphi)\to(\neg K_{b}I_{b}\varphi \to\neg I_{b}\varphi)$, various applications (A2) w.r.t. $B_{a}$ and (RMP). However, by $\mathcal{M},i\models K_{a}I_{b}\varphi$, one has also $\mathcal{M},i\models B_{a}I_{b}\varphi$, by (A6). But this is impossible by (A5). The conclusion $\mathcal{M},i\not\models \sab B_{b}\varphi$ can be reached similarly.

\item Upon noticing that $\mathcal{M},i\models I_{a}K_{b}\sab\varphi$ implies $\mathcal{M},i\models I_{a}K_{b}\varphi$ as well as $\mathcal{M},i\models I_{a}I_{a}\neg K_{b}\varphi$. However, the former condition entails that $\mathcal{M},i\models I_{a}I_{a}K_{b}\varphi$ and so $\mathcal{M},i\models I_{a}I_{a}(K_{b}\varphi\land\neg K_{b}\varphi)$, which is impossible.
\qed
\end{enumerate}
\end{proof}

Interestingly, by item (1), agent $a$ intends to preserve the secret. This is coherent with our modeling framework, which is not dynamic and cannot represent situations in which evolving dynamics make $\varphi$ not worth being protected anymore.
Item (2) states that an agent $a$ cannot keep $b$'s intentions, beliefs and knowledge secret from $b$ herself (cf. item 3 from Table \ref{tab:agot}).
Item (3) states a coherency criterion for agents' intentions. Indeed, no agent $a$ can intend to let another agent $b$ know that she is keeping a proposition $\varphi$ secret from $b$.

We observe that item (1) of Proposition \ref{prop:propertiessab2} can be generalized as follows.
Let $a\in Ag$, $\varphi\in\mathrm{Fm}_{\mathsf{S}}$ and $\Box\in\{K_{a},B_{a},I_{a}\}$. We define, for any $n\geq 0$, $\Box^{n}\varphi$ recursively as follows. We set $\Box^{0}\varphi:=\varphi$ and $\Box^{n+1}\varphi:=\Box\Box^{n}\varphi$.
\begin{proposition}\label{prop:A}Let $\varphi\in\mathrm{Fm}_{\mathsf{S}}$ and $a,b\in Ag$. Then, for any $n\geq 1$: $$\vds\sab\varphi\to I^{n}_{a}\sab\varphi.$$
\end{proposition}
    \begin{proof}
    If $a=b$, then the statement is vacuously true by Proposition \ref{prop:propertiessab}(1). Otherwise, our conclusion follows by induction applying Proposition \ref{prop:propertiessab2}(1). 
    \qed
    \end{proof}
    
It is worth observing that, while Xiong-{\AA}gotnes' notion of ``having a secret'' satisfies (at least whenever S5 is considered) what they call Secret negation completeness (see item 5 in Table \ref{tab:agot}), our notion of ``intending to keep a true secret'' does not even if negation completeness is assumed in a weaker form. Moreover, we observe that the result below shows the non-validity of item 6 from Table \ref{tab:agot} (in  \cite{Xiong2023} ensured once S5 is taken into account) in our framework.
\begin{proposition}\label{prop:notnegcompl} The following hold:
\begin{enumerate}
    \item $\not\vds K_{a}\varphi\to (\neg\sab\varphi\to K_{a}\neg\sab\varphi)$;
    \item $\not\vds K_{a}\varphi\to K_{a}\sab\varphi\lor K_{a}\neg\sab\varphi$.
\end{enumerate}
\end{proposition}
\begin{proof}Since it can be proven that (2) implies (1), we confine ourselves to provide a counterexample for (1). Let us consider the structure $$\mathcal{M}=(\{i,j,u,w\},,\{R^{I}_{a}\}_{a\in Ag},\{R^{K}_{a}\}_{a\in Ag},\{R^{B}_{a}\}_{a\in Ag},v))$$ such that, for a fixed $a\in Ag$, one has 
\begin{itemize}
    \item $R^{K}_{a}=R^{B}_{a}:=\Delta_{W}\cup\{(i,u)\}$;
    \item $R^{I}_{a}:=(\Delta_{W}\smallsetminus\{(i,i)\})\cup\{(i,j)\}$,
\end{itemize}
 and, for any other $b\in Ag$ such that $a\ne b$
 \begin{itemize}
    \item $R^{K}_{b}=R^{B}_{b}:=\Delta_{W}\cup\{(u,w)\}$;
    \item $R^{I}_{b}:=\Delta_{W}$.
\end{itemize}
Moreover, let $v$ be such that, for a fixed $p\in Var$, $v(p)=\{i,j,u\}$. It can be verified that $\mathcal{M}$ is an $\mathcal{S}$-model. Upon extending $v$ to an evaluation over $\mathrm{Fm}_{\mathsf{S}}$, one has $\mathcal{M},i\models K_{a}\varphi$, $\mathcal{M},i\models \neg I_{a}\neg K_{b}\varphi$, and so $\mathcal{M},i\models \neg\sab\varphi$. However, one has also $\mathcal{M},i\not\models K_{a}\neg\sab\varphi$.
\qed
\end{proof}

As the proof of Proposition \ref{prop:notnegcompl} shows, the Secret Negation Completeness principle fails for our operator since we do not assume any form of negative introspection w.r.t.  intentions, namely 
 \[\neg I_{a}\varphi\to K_{a}\neg I_{a}\varphi\label{naivenegintro}\tag{NI}\] need not hold. The reason why we have not assumed such a condition among our assumptions is that negative introspection w.r.t. intentions seems a rather strong requirement as it presupposes that any agent has always a clear position concerning whether she intends to bring about states of affairs validating a proposition or not, for \emph{any} proposition. Indeed, while, according to  \cite{Slepian2017}, ``an
individual has a secret the moment he or she decides to withhold information about an episode or act from another person'', so intentions of secrecy are always explicit and known, agents may not be aware that they do not intend to keep a secret, think e.g. to cases of self-deception, or the more common cases in which agents do not even take into account a given proposition.

Beside negative results from Proposition \ref{prop:notnegcompl}, the next proposition highlights further ``non-properties'' of our operator due to the lack of negative introspection, and our definition of ``intending to keep a true secret''.

\begin{proposition}\label{proposition:nonproperties}
    Let $\varphi\in\mathrm{Fm}_{\mathsf{S}}$ and $a,b\in Ag$ with $a\ne b$. Then:
\begin{enumerate}
\item $\not\vds\neg\sab\varphi\to \sab\neg \sab\varphi$; 
\item
$\not\vds\neg \sab\varphi\to \neg \sab\neg \sab\varphi$;
\item $\not\vds\neg \sab(\varphi\to\psi)\to(\neg \sab\varphi\to\neg \sab\psi)$;
\item $\not\vds\neg \sab\neg K_{b}\varphi$;
\item $\not\vds\neg\sab\neg S_{b,a}\varphi$
\end{enumerate}
\end{proposition}

Note that (1)-(3) from  Proposition \ref{proposition:nonproperties} fail in \cite{Xiong2023} even whenever S5 is considered (item 17, 23, 25 from Table \ref{tab:agot}). In turn, this entails that our system witnesses that items 26 and 27 from Table \ref{tab:agot} hold in our framework as well.  However, due to the definition they give, $S_{a}$ satisfies (4) and (5) (see item 4 from Table \ref{tab:agot}).

\section{Relation between true secrets and information flows}\label{sec:relation between}

Perhaps the main advantage of representing secrecy by combining concepts such as knowledge, belief, and intentions is that it allows us to formally reason about issues concerning, on the one hand, the ability of the keeper of the secret to keep it a secret and, on the other hand, the ability of the nescient to learn anything about it. These issues are studied extensively by theories devoted to analyzing information flows within complex systems (e.g., networks of cyber-physical systems and groups of human agents following the steps of an interactive communication protocol). Two remarkable examples are given by the principle of Zero-Knowledge (ZK)~\cite{10.1137/0218012} and by the noninterference approach to confidentiality analysis~\cite{GM82}. The former is concerned with the ability of the keeper to prove to be the secret keeper without revealing any further information about the secretum other than the keeper's knowledge of it (e.g., in specific security protocols, the keeper is interested in showing to possess a given key without revealing it). The latter defines a security model employed to restrict the information flow through systems and is based on verifying whether (the use of) private information has a side effect on (the observations of) public information.

In this section, we present properties of the secrecy operator that help formalize some of the issues about information flow analysis in a logical setting by proposing examples related to ZK and noninterference. In particular, we will discuss the relationship among secrets, the intentions of the secret keeper, and the consequences for the nescient.

We start with the remark that, in the present framework, the Interpolation Rule still holds, cf. p. \pageref{eq:intprule}.

\begin{proposition} 
\label{prop:implications}
Let $\varphi,\psi \in\mathrm{Fm}_{\mathsf{S}}$ and
$a,b\in Ag$ with $a\ne b$. The following hold:
\begin{enumerate}
    \item $\vds \varphi\to\psi$ implies $\vds B_{a}\neg K_{b}\psi\to B_{a}\neg K_{b}\varphi$;
    \item $\vds \varphi\to\psi$ implies $\vds I_{a}\neg K_{b}\psi\to I_{a}\neg K_{b}\varphi$;
    \item $\vds\varphi\to\psi$ and $\vds\psi\to\chi$ imply $\vds(\sab\varphi\land\sab\chi)\to\sab\psi$.
\end{enumerate}
\end{proposition}

We believe that a further deepening of the meaning of the Xiong-{\AA}gotnes Interpolation principle is in order. For example, suppose that ``uploading a top-secret document to a private repository'' ($\varphi$) implies that ``the repository turns on read-only access'' ($\psi$), and that $\psi$ implies that ``an alert is sent to the repository admin'' ($\chi$). Then, suppose both $\varphi$ and $\chi$ are kept secret from an agent $b$ (even if $b$ is aware of the kind of connection from $\varphi$ to $\chi$ through $\psi$). In that case, it is reasonable to assume that $b$ cannot be aware of $\psi$ under the hypothesis above because no observation can help $b$ to infer the veracity of $\psi$.\\
It is worth observing that Proposition~\ref{prop:implications}(3) can be extended to inference chains of arbitrary length.
\begin{proposition}Let $\{\varphi_{1},\dots,\varphi_{n}\}\subseteq\mathrm{Fm}_{\mathsf{S}}$ ($n\geq 2$) and $a,b\in Ag$. Then
\[\vds\varphi_{i}\to\varphi_{i+1}\ (1\leq i\leq n)\text{ implies }\vds\sab\varphi_{1}\land\sab\varphi_{n}\to\bigwedge_{i=1}^{n}\sab\varphi_{i}.\]
\end{proposition}
\begin{proof}
By induction on $n$. The case $n=2$ is trivial. Now, assume that our claim holds for $n\geq 2$. By axioms of classical logic, one has $\vds\varphi_{1}\to\varphi_{n}$ which, jointly with $\vds \varphi_{n}\to\varphi_{n+1}$, implies, by Proposition~\ref{prop:implications}(3), $\vds \sab\varphi_{1}\land\sab\varphi_{n+1}\to\sab\varphi_{n}$. Also, a customary argument shows that $\vds \sab\varphi_{1}\land\sab\varphi_{n+1}\to\sab\varphi_{1}\land \sab\varphi_{n}$. Now, by induction hypothesis one has $\sab\varphi_{1}\land \sab\varphi_{n}\to\bigwedge_{i=1}^{n}\sab\varphi_{i}$. By the transitivity of $\to$, $\vds \sab\varphi_{1}\land\sab\varphi_{n+1}\to\bigwedge_{i=1}^{n}\sab\varphi_{i}$ and so, reasoning as above, $\vds \sab\varphi_{1}\land\sab\varphi_{n+1}\to\bigwedge_{i=1}^{n+1}\sab\varphi_{i}$. \qed
\end{proof}

We observe that an analogous result can be obtained also for the Xiong-{\AA}gotnes operator by putting in good use the same inductive argument and properties stated in Table \ref{tab:agot}.
We now reason about the secrecy of implication.

\begin{proposition}\label{prop:pierluigigraziani}Let $\varphi,\psi \in\mathrm{Fm}_{\mathsf{S}}$ and $a,b\in Ag$ with $a\ne b$. The following hold:
\begin{enumerate}
     \item $\vds B_{a}K_{b}\psi\lor I_{a}K_{b}\psi\to\neg\sab(\varphi\to\psi);$
     \item $\vds B_{a}K_{b}\neg\varphi\lor I_{a}K_{b}\neg\varphi\to\neg\sab(\varphi\to\psi).$
\end{enumerate}
\end{proposition}

\begin{proof}
\mbox{}
\begin{enumerate}
\item Let $\mathcal{M}=\mathcal{M}=(W,\{R^{I}_{a}\}_{a\in Ag},\{R^{K}_{a}\}_{a\in Ag},\{R^{B}_{a}\}_{a\in Ag},v)$ be an arbitrary $\mathcal{S}$-model and let $i\in W$. Assume that $\mathcal{M},i\models B_{a}K_{b}\psi\lor I_{a}K_{b}\psi$. Also, suppose towards a contradiction that $\mathcal{M},i\models B_{a}\neg K_{b}(\varphi\to\psi)$. Let us first consider the case $\mathcal{M},i\models B_{a}K_{b}\psi$. An easy check shows that $\mathcal{M}i\models B_{a}(K_{a}\psi\land \neg K_{b}(\varphi\to\psi))$. However, one has also that $\vds\psi\to(\varphi\to\psi)$ and so, by ($\mathrm{RN}_{K_{b}}$), (A2) and (RMP) $\vds K_{b}\psi\to K_{b}(\varphi\to\psi)$. Therefore, by ($\mathrm{RN}_{B_{a}}$), $\vds B_{a}(K_{b}\psi\to K_{b}(\varphi\to\psi))$. But then $\mathcal{M},i\models B_{a}(K_{a}\psi\land \neg K_{b}(\varphi\to\psi))\land B_{a}(K_{b}\psi\to K_{b}(\varphi\to\psi))$, which is a contradiction by (A6). Assuming that $\mathcal{M},i\models I_{a}K_{b}\psi$, and applying analogous reasoning, we reach the same conclusion by (A10).
\item The result follows from item 1 and from $\vds(\varphi\to\psi)\leftrightarrow(\neg\psi\to\neg\varphi)$. \qed
\end{enumerate}
\end{proof}

The intuition, based on the semantics of $\to$, is that the secrecy of $\varphi\to\psi$ decays when agent $b$ is supposed to know 
(or it is intended agent $b$ to know) either $\psi$ -- see item ($1$) -- or the negation of $\varphi$ -- see item ($2$). Note that, due to the way \cite{Xiong2023} formalizes it, $S_{a}$ satisfies both conditions of Proposition \ref{prop:pierluigigraziani}.
Inspired by the same line of reasoning, we emphasize the effect of conjunctions of information over secrets in the following results.

\begin{proposition}\label{prop:conjunctions}
Let $\varphi,\psi \in\mathrm{Fm}_{\mathsf{S}}$ and $a,b\in Ag$ with $a\ne b$. The following hold:
\begin{enumerate}
\item $\vds \sab\varphi\land K_{a}\psi\land I_{a}\psi\to\sab(\varphi\land\psi)$;
\item $\vds\sab\varphi\land\sab\psi\to\sab(\varphi\land\psi)$;
\item $\vds \sab(\varphi\land\psi)\land\sab(\varphi\lor\psi)\to\sab\varphi\land\sab\psi$.
\end{enumerate}
\end{proposition}
\begin{proof}
Let $\mathcal{M}=(W,\{R^{I}_{a}\}_{a\in Ag},\{R^{K}_{a}\}_{a\in Ag},\{R^{B}_{a}\}_{a\in Ag},v)$ be an arbitrary $\mathcal{S}$-model, and $i \in W$. Then:
\begin{enumerate}
    \item 
Note that $\mathcal{M},i\models\sab\varphi\land K_{a}\psi$ implies $\mathcal{M},i\models K_{a}\varphi\land K_{a}\psi$, and so also $\mathcal{M},i\models K_{a}(\varphi\land\psi)$. Moreover, it is easily seen that $\vds B_{a}\neg K_{b}\varphi\to B_{a}\neg K_{b}(\varphi\land\psi)$ as well as $\vds I_{a}\neg K_{b}\varphi\to I_{a}\neg K_{b}(\varphi\land\psi)$. Finally, since $\mathcal{M},i\models I_{a}\varphi\land I_{a}\psi$, one has $\mathcal{M},i\models I_{a}(\varphi\land\psi)$. Therefore. $\mathcal{M},i\models\sab\varphi\land K_{a}\psi\land I_{a}\psi\to K_{a}(\varphi\land\psi)\land B_{a}\neg K_{b}(\varphi\land\psi)\land I_{a}(\neg K_{b}(\varphi\land\psi)\land(\varphi\land\psi))=\sab(\varphi\land\psi)$. 
   \item 
The result follows directly by the previous item upon noticing that $\vds\sab\psi\to K_{a}\psi$. 
   \item 
The result follows by Proposition~\ref{prop:implications}(3) upon noticing that $\vds\varphi\land\psi\to\varphi$ and $\vds\varphi\land\psi\to\psi$ as well as $\vds\varphi\to\varphi\lor\psi$ and $\vds\psi\to\varphi\lor\psi$. \qed

\end{enumerate}
\end{proof}

The first two items express conditions implying a form of additivity of secrecy for conjunction. 
Regarding item ($1$), the secrecy of $\varphi$ is pursued by agent $a$ even if $\varphi$ is in conjunction with any other piece of information $\psi$ known (and intended to stay true) by agent $a$, otherwise revealing $\varphi \land \psi$ could expose $\varphi$ to agent $b$. A stronger result without the third conjunct is provable in \cite{Xiong2023}. Item ($2$) expresses a much stronger condition about $\psi$, which, for the previous case, is kept secret by agent $a$. However, it is worth noticing that the opposite direction for item ($2$) does not hold. Those results correspond to item 19 and 20 in Table \ref{tab:agot}. A weaker result concerning such a direction is exemplified in item ($3$), which could also be obtained in \cite{Xiong2023}, since this is an application of the interpolation rule. Hence, the secrecy of $(\varphi \land \psi)$, by itself, is insufficient to conclude anything about the secrecy of its constituents, as shown by the following result.

\begin{proposition}\label{prop:nonmonotonic}Let $\varphi,\psi \in\mathrm{Fm}_{\mathsf{S}}$ and $a,b\in Ag$ with $a\ne b$. The following holds:
$$\not\vds S_{a,b}(\varphi\land\psi)\to \sab\varphi\lor\sab\psi.$$
\end{proposition}
\begin{proof}
Let us consider the structure $$\mathcal{M}=(\{i,i_{1},i_{2},i_{3},i_{4}\},\{R^{I}_{a}\}_{a\in Ag},\{R^{K}_{a}\}_{a\in Ag},\{R^{B}_{a}\}_{a\in Ag},v)$$ such that, for a fixed $a\in Ag$, one has
    \begin{itemize}
        \item $R^{K}_{a}=\Delta_{W}\cup\{(i,i_{1}),(i,i_{2}),(i_{1},i_{2}),(i_{2},i_{1})\}$;
        \item $R^{B}_{a}=R^{I}_{a}=\Delta_{w}\smallsetminus\{(i,i)\}\cup\{(i,i_{1}),(i,i_{2}),(i_{1},i_{2}),(i_{2},i_{1})\}$.
        \end{itemize}
and for any $b\ne a$
\begin{itemize}
        \item $R^{K}_{b}=R^{I}_{b}=\Delta_{W}\cup\{(i_{2},i_{4}),(i_{4},i_{2}),(i_{1},i_{3}),(i_{3},i_{1})\}$;
        \item $R^{B}_{b}=\Delta_{W}\cup\{(i_{2},i_{4}),(i_{1},i_{3})\}$.
    \end{itemize}
Also, let $v(p)=\{i,i_{1},i_{2},i_{3}\}$ and $v(q)=\{i,i_{1},i_{2},i_{4}\}$. Upon extending $v$ to an evaluation over $\mathrm{Fm}_{\mathsf{S}}$, one has $\mathcal{M},i\models\sab(p\land q)$. However, since $M,i_{2}\models K_{b}p$ and $M,i_{1}\models K_{b}p$ and $M,i_{2}\models K_{b}q$, one has that $\mathcal{M},i\not\models B_{a}\neg K_{b}p$ as well as $\mathcal{M},i\not\models B_{a}\neg K_{b}q$. Consequently, our result was obtained by definition. \qed
\end{proof}

The intuition behind such a negative result is that even though agent $a$ believes that $(\varphi \land \psi)$ is unknown to agent $b$, agent $a$ has no reason to believe that $\varphi$ (or $\psi$), alone, is unknown to agent $b$, which is a condition necessary to determine the secrecy of $\varphi$ (or of $\psi$). The same intuition applies to the intention $I_a(\neg K_b (\varphi \land \psi))$, from which we can derive neither $I_a(\neg K_b \varphi)$
nor $I_a(\neg K_b \psi)$. In other words, agent $a$ cannot decide which among $\varphi$ and $\psi$ is worth hiding to agent $b$. Note that Proposition   \ref{prop:nonmonotonic} shows, among other things, that the $\sab$ operator is non-monotonic, therefore mirroring an analogous property of the $S_{a}$ operator (see item 24 from Table \ref{tab:agot}).

Regarding information disjunction, we have the following results related to secrecy. 

\begin{proposition}
Let $\varphi,\psi\in\mathrm{Fm}_{\mathsf{S}}$ and $a,b\in Ag$ with $a\ne b$. The following hold:
\begin{enumerate}
\item $\vds\sab(\varphi\lor\psi)\land K_{a}\varphi\land I_{a}\varphi\to\sab\varphi$;
\item $\vds \sab(\varphi\lor\psi) \land \sab(\neg\varphi) \to \sab\psi$;
\item $\not\vds\sab(\varphi\lor\psi)\to\sab\varphi\lor\sab\psi$;
\item $\not\vds\sab\varphi\lor\sab\psi\to\sab(\varphi\lor\psi)$.
\end{enumerate}
\end{proposition}
    
As opposed to the case of additivity of conjunction, item ($1$) expresses a form of contraction of secrecy for disjunction. If agent $a$ intends to preserve the secrecy of $\varphi \lor \psi$ and simultaneously intends to maintain the veracity of $\varphi$, then she must preserve the secrecy of $\varphi$. A stronger result without the third conjunct can be easily proven in \cite{Xiong2023}.
Item ($2$) elaborates more on the effect of a secret disjunction.
As an illustrative example, let us consider the following home-banking access scenario. If an agent $a$ intends to keep her access policy ``using a strong password'' ($\varphi$) or ``using a 2-factor authentication system'' ($\psi$) secret from any adversary agent $b$ -- as formally represented by $\sab(\varphi\lor\psi)$ -- and, moreover, agent $a$ intends to keep ``using a weak password'' ($\neg\varphi$) secret from $b$ -- formally, $\sab(\neg\varphi$) -- one must conclude that $a$ is keeping ``using a 2-factor authentication system'' secret from $b$, namely $\sab\psi$. Item (2) can easily be obtained in \cite{Xiong2023}, see item 18 from Table \ref{tab:agot}. Moreover, in \cite{Xiong2023}, similar counter-models can be built for items (3) and (4).

As anticipated at the beginning of Section~\ref{sec:relation between},
we now use our logical framework to reason about important principles based on any ZK proof~\cite{10.1137/0218012}. Let us consider the following result.

\begin{proposition}\label{prop:ZKproof}
Let $\varphi\in\mathrm{Fm}_{\mathsf{S}}$ and $a,b\in Ag$ with $a\ne b$. Then:
\[ \not\vds \sab\varphi\land K_{a}(\varphi\to\psi)\land K_{b}\psi\to K_{b}\varphi. \]
\end{proposition}

\begin{proof}
Let us consider the $\mathcal{S}$-model 
\[\mathcal{M}=(W=\{i,j,k,w\}.\{R^{I}_{a}\}_{a\in Ag},\{R^{K}_{a}\}_{a\in Ag},\{R^{B}_{a}\}_{a\in Ag},v)\] such that, for a fixed $a\in Ag$,
\begin{itemize}
    \item $R^{K}_{a}:=\Delta_{W}\cup\{(i,j)\}$;
    \item $R^{I}_{a} = R^{B}_{a}:=R^{K}_{a}\smallsetminus\{(i,i)\}$;
\end{itemize}
and, for any $b\ne a$,
\begin{itemize}
    \item $R^{K}_{b}:=\Delta_{W}\cup\{(i,j),(j,w)\}$;
    \item $R^{I}_{b}=R^{B}_{b}:= R^{K}_{b}\smallsetminus\{(i,i)\}$.
\end{itemize}
A routine check shows that $(W=\{i,j,k,w\}.\{R^{I}_{a}\}_{a\in Ag},\{R^{K}_{a}\}_{a\in Ag},\{R^{B}_{a}\}_{a\in Ag})$ is indeed an $\mathcal{S}-$frame. Moreover, fix $p,q\in Var$ and let $v:Var\to\mathcal{P}(W)$ be such that $v(p)=\{i,j\}$ and $v(q)=\{i,j,k\}$. Upon extending $v$ to an evaluation on the whole $\mathrm{Fm}_{\mathsf{S}}$, one has that $\mathcal{M},i\models K_{a}(p\land q)$ and so also $\mathcal{M},i\models K_{a}(p\to q)\land\sab p\land K_{b}q$. However, $\mathcal{M},i\not\models p$, since $\mathcal{M},{k}\models\neg p$.  \qed
\end{proof}

We illustrate the underlying intuition through an example from the setting of public key cryptography, which is a typical application domain of ZK proofs. Suppose that, as is common practice, the private key for the digital signature of agent $a$ is kept secret from any other agent $b$ ($\sab\varphi$). Agent $a$ knows that such a key is used to generate a valid signature on a particular document ($K_{a}(\varphi\to\psi)$). The ZK proof based on the use of the public key of agent $a$ allows agent $b$ to know that such a document is indeed signed by agent $a$ ($K_{b}\psi$). However, in no way does agent $b$ learn anything about the secret key of agent $a$. Therefore, $K_{b}\varphi$ cannot be a consequence of the previous conditions. Again, a counter-model for the formula of Proposition \ref{prop:ZKproof} can also be constructed in \cite{Xiong2023}.

Now, we investigate how our logical framework can be employed to reason about the noninterference approach to information flow analysis. First, let us observe a preliminary property of the secrecy operator. 
Since $S_{a,b}\varphi$ does not establish any certainty about 
the actual secrecy of $\varphi$, it is compatible with scenarios in which agent $a$ intends to keep $\varphi$ secret from agent $b$ and, at the same time, agent $b$ intends to keep $\varphi$ secret from agent $a$.
Consequently, we also have that $S_{a,b}\varphi$ is compatible even with a scenario in which, by any circumstances, agent $b$ acquires the knowledge of $\varphi$ despite the efforts of agent $a$. The following proposition formalizes the compatibilities above showing, among other things, the failure of items 1 and 8 from Table \ref{tab:agot}.

\begin{proposition}\label{prop:notperfsecret} Let $\varphi\in\mathrm{Fm}_{\mathsf{S}}$ and $a,b\in Ag$ with $a\ne b$. Then, the following hold:
\begin{enumerate}
\item $\not\vds S_{a,b}\varphi\to\neg S_{b,a}\varphi$;
    \item $\not\vds S_{a,b}\varphi\to\neg K_{b}\varphi$.
\end{enumerate}
\end{proposition}
\begin{proof}
First, note that (2) follows by (1). Then, let us consider the following $\mathcal{S}$-model $$\mathcal{M}=(W=\{i,j,w,u,k\},\{R^{I}_{a}\}_{a\in Ag},\{R^{K}_{a}\}_{a\in Ag},\{R^{B}_{a}\}_{a\in Ag},v)$$ such that 
\begin{itemize}
    \item $R^{K}_{a}=\Delta_{W}\cup\{(i,j),(w,k)\}$;
    \item $R^{B}_{a}=R^{I}_{a}:=\{(i,j),(j,j),(u,u),(w,k),(k,k)\}$,
\end{itemize}
and, for any $b\in Ag\smallsetminus\{a\}$:
\begin{itemize}
    \item $R^{K}_{b}:=\Delta_{W}\cup\{(i,w),(j,u)\}$;
    \item $R^{B}_{b}=R^{I}_{b}=\{(i,w),(w,w),(k,k),(u,u),(j,u)\}$.
\end{itemize}
Let $v$ be such that $v(p)=\{i,j,w\}$. Upon extending $v$ to an evaluation over the whole $\mathrm{Fm}_{\mathsf{S}}$, it can be verified that $\mathcal{M},i\models\sab\varphi$ as well as $\mathcal{M},i\models S^{*}_{b,a}\varphi$. \qed
\end{proof}

Those results clearly hold also in \cite{Xiong2023} whenever factivity for the knowledge operator is dropped.

As previously mentioned, the objective of the noninterference theory~\cite{GM82}
is to study the effect of indirect information flows causing the disclosure of secret information. The flows of interest are those that can be formalized through formulas like:
\[S_{a,b}\psi \land K_b \varphi \land K_b (\varphi \to \psi) \tag{IIF}\label{indirectIF}
\]
from which it is easy to see that $K_b\psi$ immediately derives.

The sub-formula $S_{a,b}\psi$ expresses, in a way, the intentions and expectations of agent $a$ about the secrecy of $\psi$ for agent $b$. Hence, $\psi$ is the secretum that agent $a$ does not reveal directly to agent $b$.
The sub-formula $K_b (\varphi \to \psi)$ expresses a relation (between $\varphi$ and $\psi$) known to agent $b$ and triggering, in combination
with the sub-formula $K_b \varphi$, the indirect information flow that allows agent $b$ to infer  $\psi$ through the knowledge of $\varphi$. 
In other words, the knowledge of $(\varphi \to \psi)$ is at the base of the interference that, given $K_b \varphi$, provides knowledge of the secret.
Formally, it is worth observing that formula~(\ref{indirectIF}) is satisfiable in $\mathsf{S}$. However, the same does not hold for the Xiong-{\AA}gotnes' operator $S$, at least if the factivity of knowledge is assumed. This highlights that, whenever `` true secrets'' are taken into account, the notion of ``having a secret'' portrayed by $S$ results in an actual impossibility for outsiders to infer the content of the secret, namely, secret keepers have full control over the secrets they have. On the contrary, intentions of secrecy do not exclude possible information leakages due to the inferential behavior of nescients. This is a fundamental point for information flow analysis and the modeling of noninterference\footnote{This is in contrast with item 1 of Table \ref{tab:agot}, which is the very definition of secret given in \cite{Xiong2023}.}.

The following line of pseudo-code represents a classic example of \eqref{indirectIF}: 
\[
\mathtt{if~secret\_boolean} == 0: \mathtt{print(``Hello, World!")} 
\]
where the value of $\mathtt{secret\_boolean}$ is the secret $\psi$ that cannot be directly accessed by agent $b$, who can, however, observes the standard output. Hence, by observing the output $\mathtt{``Hello, World!"}$ ($K_b \varphi$) and by knowing the program ($K_b(\varphi \to \psi$)), agent $b$ infers information about the secret  without reading it directly. 

Revealing to agent $a$ the nature of such an information flow known by agent $b$ is preparatory to finding a solution to the information leakage. The following result expresses such an idea.

\begin{proposition}\label{prop:rev1}
Let $\varphi,\psi \in\mathrm{Fm}_{\mathsf{S}}$ and $a,b\in Ag$ with $a\ne b$. \\ Then, $\vds K_{a}(K_{b}\varphi\land K_{b}(\varphi\to\psi))\to K_a (\neg\sab\psi).$
\end{proposition}

Intuitively, if agent $a$ is aware that agent $b$ knows the indirect information flow, then she realizes that the secrecy of $\psi$ is violated.
We stress that what matters most is not agent $a$'s knowledge of the information flow but agent $a$'s knowledge/belief about agent $b$'s knowledge. Indeed, it can be seen that the same results hold for the Xiong-{\AA}gotnes operator $S$, once the S4 framework is considered. Moreover, the dependence of secrecy's intentions from actual secret keepers' beliefs about nescients knowledge is also emphasized by the following negative result, whose validity is also preserved in \cite{Xiong2023}.

\begin{proposition}\label{prop:rev2}
Let $\varphi,\psi \in\mathrm{Fm}_{\mathsf{S}}$ and $a,b\in Ag$ with $a\ne b$. \\ Then, $\not\vds\sab\psi \land K_{a}(\varphi\to\psi) \,\land\, K_a \varphi \to \sab\varphi$.
\end{proposition}

\begin{proof}
    Let us consider the structure $$\mathcal{M}=(\{i,j,w,r\},\{R^{I}_{a}\}_{a\in Ag},\{R^{K}_{a}\}_{a\in Ag},\{R^{B}_{a}\}_{a\in Ag},v)$$ such that, for a fixed $a\in Ag$, one has
    \begin{itemize}
        \item $R^{K}_{a}:=\Delta_{W}\cup\{(i,j),(j,w),(i,w)\}$;
        \item $R^{B}_{a}=R^{I}_{a}:=\{(i,w),(j,w),(w,w),(r,r)\},$
        \end{itemize}
and for any $b\ne a$
\begin{itemize}
        \item $R^{K}_{b}=R^{B}_{b}=R^{I}_{b}=\Delta_{W}\cup\{(w,r)\}$.
    \end{itemize}
Moreover, let $v$ be such that for fixed $p,q\in \mathit{Var}$, $v(p)=W$ and $v(q)=\{i,j,w\}$. A customary check shows that $\mathcal{M}$ is indeed an $\mathcal{S}$-model. Moreover, upon extending $v$ to an evaluation $v$ over $\mathrm{Fm}_{\mathsf{S}}$, one has that $\mathcal{M},i\models(\sab q)\land K_{a}(p\to q)$ but $\mathcal{M},i\not\models\sab p$.   \qed
\end{proof}

Intuitively, even if agent $a$ is aware of an information flow from $\varphi$ to $\psi$ and $\psi$ is a secret protected by agent $a$, not necessarily agent $a$ intends to protect also $\varphi$ from agent $b$.
The reason could be that agent $a$ has good reasons to believe that agent $b$ does not know the information flow from $\varphi$ to $\psi$.
This kind of reasoning is quite common in strategic risk assessment.
For instance, a chief information security officer (CISO) $a$ may be aware of a vulnerability in the corporate network ($\varphi$), which could lead to a data breach ($\psi$). However, even if the information loss is undesired ($S_{a,b}(\psi)$), CISO $a$ may decide not to protect the network, e.g., because it is supposed that hackers do not know how to exploit it. Under these considerations, it is interesting to formalize the intentions of the CISO in the unfortunate case in which it is believed that the hackers may exploit the vulnerability. The following result provides valuable insights. 

\begin{proposition}\label{prop:cisoex}
Let $\varphi,\psi \in\mathrm{Fm}_{\mathsf{S}}$ and $a,b\in Ag$ with $a\ne b$. Then:
$$\vds\sab\psi\land B_{a}K_{b}(\varphi\to\psi)\land K_{a}\varphi\to[\neg \sab\varphi \to (\neg I_{a}\varphi \lor \neg I_{a}K_{b}(\varphi\to\psi))].$$
\end{proposition}

\begin{proof}
Let us consider the contrapositive of the consequent of the stated formula: $I_{a}\varphi \land I_{a}K_{b}(\varphi\to\psi)\to \sab\varphi.$
Let $\mathcal{M}$ be an arbitrary $\mathcal{S}$-model and let $i\in W$. Assume that $\mathcal{M},i\models \sab\psi\land B_{a}K_{b}(\varphi\to\psi)\land K_{a}\varphi$ but $\mathcal{M},i\models I_{a}\varphi\land I_{a}K_{b}(\varphi\to\psi)\land\neg\sab\varphi.$ Note that $\mathcal{M},i\models B_{a}K_{b}(\varphi\to\psi)$ and $\mathcal{M},i\models B_{a}\neg K_{b}\psi$ entails $\mathcal{M},i\models B_{a}\neg K_{b}\varphi$. Furthermore, $\mathcal{M},i\models I_{a}K_{b}(\varphi\to\psi)$ implies $\mathcal{M},i\models I_{a}\neg K_{b}\psi\to I_{a}\neg K_{b}\varphi$. In view of our assumptions, we conclude that $\mathcal{M},i\models I_{a}\neg K_{b}\varphi$. This boils down to a contradiction: $\mathcal{M},i\models\sab\varphi$. \qed
\end{proof}

Going back to our example, CISO $a$ has to protect sensitive data ($S_{a,b}(\psi)$), is aware of the vulnerability ($K_{a}\varphi$), and believes that the hackers know how to exploit the vulnerability
($B_{a}K_{b}(\varphi\to\psi)$). Then, if the CISO aims to avoid hiding the vulnerability from the hackers ($\neg S_{a,b}(\varphi)$), two alternative solutions are possible:
\begin{enumerate}
    \item the CISO invests to resolve the vulnerability ($\neg I_{a}\varphi$);
    \item the CISO invests to evade or mitigate attacks that can exploit the vulnerability ($\neg I_{a}K_{b}(\varphi\to\psi)$).
\end{enumerate}

Proposition~\ref{prop:cisoex} and the CISO example emphasize the kind of complex real-world scenarios about secrecy intentions that our operator is able to describe.

Until now, we have dealt with situations involving no more than two agents: a secret keeper and an alleged nescient. The following proposition provides some results concerning situations involving three agents and the transfer of secrets among them. 

\begin{proposition}\label{prop:threeagents}Let $\varphi \in\mathrm{Fm}_{\mathsf{S}}$ and $a,b,c\in Ag$ any triple of mutually distinct agents. The following hold:
\begin{enumerate}
\item $\vds K_{c}\sab\varphi\to K_{c}\varphi$;
\item $\vds\sab\varphi\to I_{a}\neg K_{b}S_{c,b}\varphi$;
\item $\vds\sab\varphi\to [B_{a}(K_{c}\varphi\to K_{b}\varphi)\land(I_{a}(K_{c}\varphi\to K_{b}\varphi))\leftrightarrow S_{a,c}\varphi]$;
\item $\vds\sab\varphi \land  I_{a}K_{c}\varphi\to I_{a}\neg(K_{c}\varphi\to K_{b}\varphi)$;
\item $\vds\sab\varphi\land K_{a}S_{c,b}\varphi\to( I_{a}S_{c,b}\varphi\leftrightarrow\sab (S_{c,b}\varphi))$;
\item $\not\vds\sab\varphi\land K_{a}S_{c,b}\varphi\to\sab (S_{c,b}\varphi)$;
\item $\not\vds\sab(S_{c,b}\varphi)\to\sab\varphi$;
\end{enumerate}    
\end{proposition}
\begin{proof} Let $\mathcal{M}$ be an arbitrary $\mathcal{S}$-model and $i \in W$.

\begin{enumerate}
\item Direct consequence of Proposition~\ref{prop:propertiessab}(3).
\item By (A8), it is sufficient to show that $\mathcal{M},i\models \sab\varphi \land I_{a} K_{b}S_{c,b}\varphi$ is a contradiction. From $\mathcal{M},i\models \sab\varphi$ we derive $\mathcal{M},i\models I_{a} \neg K_{b} \varphi$ and from
$\mathcal{M},i\models I_{a} K_{b}S_{c,b}\varphi$ we derive 
$\mathcal{M},i\models I_{a} K_{b} \varphi$, thus the contradiction.

\item Suppose that $\mathcal{M},i\models\sab\varphi\land K_{a}S_{c,b}\varphi\land I_{a}S_{c,b}\varphi$. We just need to prove that $\mathcal{M},i\models B_{a}\neg K_{b}S_{c,b}\varphi$ and $\mathcal{M},i\models I_{a}\neg K_{b}S_{c,b}\varphi$. Since the latter follows from (2), we focus on the former. Suppose towards a contradiction that $\mathcal{M},i\not\models B_{a}\neg K_{b}S_{c,b}\varphi$. Therefore, there is $j\in W$ such that $R^{B}_{a}(i,j)$ and $\mathcal{M},j\models K_{b}S_{c,b}\varphi$. The latter condition entails $\mathcal{M},j\models K_{b}\varphi$. However, this means that $\mathcal{M},i\not\models B_{a}\neg K_{b}\varphi$, a contradiction. So we have, $\mathcal{M},i\models\sab\varphi\land K_{a}S_{c,b}\varphi\to(I_{a}S_{c,b}\varphi\to \sab(S_{c,b}\varphi)$. Therefore, since $\vds \sab(S_{c,b}\varphi)\to I_{a}S_{c,b}\varphi$, our conclusion obtains.

\item Assume that $\mathcal{M},i\models\sab\varphi$. If $\mathcal{M},i\models B_{a}(K_{c}\varphi\to K_{b}\varphi)\land I_{a}(K_{c}\varphi\to K_{b}\varphi)$, then one has $\mathcal{M},i\models B_{a}\neg K_{b}\varphi\to B_{a}\neg K_{c}\varphi$ as well as $\mathcal{M},i\models I_{a}\neg K_{b}\varphi\to I_{a}\neg K_{c}\varphi$.  Consequently, since $\mathcal{M},i\models\sab\varphi$, we have also $\mathcal{M},i\models B_{a}\neg K_{c}\varphi$ as well as $\mathcal{M},i\models I_{a}\neg K_{c}\varphi$. So, $\mathcal{M},i\models S_{a,c}\varphi$. Conversely, if $\mathcal{M},i\models S_{a,c}\varphi$, then $\mathcal{M},i\models I_{a}\neg K_{c}\varphi$. However, since $\mathcal{M},i\models I_{a}\neg K_{c}\varphi\to I_{a}(\neg K_{b}\varphi\to \neg K_{c}\varphi)$, we have also $\mathcal{M},i\models I_{a}(\neg K_{b}\varphi\to \neg K_{c}\varphi)$. Therefore, we have $\mathcal{M},i\models I_{a}(K_{c}\varphi\to K_{b}\varphi)$. Similarly, we prove that $\mathcal{M},i\models B_{a}(K_{c}\varphi\to K_{b}\varphi)$. 

\item Assume $\mathcal{M},i\models \sab\varphi \land I_{a}K_{c}\varphi$. If $\mathcal{M},i\not\models I_{a}\neg (K_{c}\varphi\to K_{b}\varphi)$, one has that there is $j\in W$ such that $R^{I}_{a}(i,j)$ and $\mathcal{M},j\models K_{c}\varphi\to K_{b}\varphi$. Therefore, $\mathcal{M},j\models K_{b}\varphi$ and so $\mathcal{M},i\not\models I_{a}\neg K_{b}\varphi$, a contradiction.

\item Let us consider the structure $$\mathcal{M}=(\{i,j,w,r,u,v\},\{R^{I}_{a}\}_{a\in Ag},\{R^{K}_{a}\}_{a\in Ag},\{R^{B}_{a}\}_{a\in Ag},v)$$ such that, for fixed $a,c\in Ag$ ($a\ne c$), one has
    \begin{itemize}
        \item $R^{K}_{a}:=\Delta_{W}\cup\{(i,j),(j,w),(i,w)\}$;
        \item $R^{B}_{a}=R^{I}_{c}=R^{B}_{c}:=\{(i,w),(j,w),(w,w),(r,r)\}$;
        \item $R^{I}_{a}:=R^{B}_{a}\cup\{(i,u)\}$;
        \item $R^{K}_{c}=R^{K}_{a}\cup\{(u,v)\}$;
        \end{itemize}
and for any $b\ne a,c$
\begin{itemize}
        \item $R^{K}_{b}=R^{B}_{b}=R^{I}_{b}=\Delta_{W}\cup\{(w,r)\}$.
    \end{itemize}
Moreover, let $v$ be such that for a fixed $q\in Var$, $v(q)=\{i,j,w,u\}$. A customary check shows that $\mathcal{M}$ is indeed an $\mathcal{S}$-model. Moreover, upon extending $v$ to an evaluation $v$ over $\mathrm{Fm}_{\mathsf{S}}$, one has that $\mathcal{M},i\models\sab q\land K_{a}S_{c,b}q$ but $\mathcal{M},i\not\models\neg I_{a}K_{c}q$, an so $\mathcal{M},i\not\models\sab(S_{c,b}q)$ as well.

\item Let us consider the structure  $$\mathcal{M}=(\{i,j,w,u\},\{R^{I}_{a}\}_{a\in Ag},\{R^{K}_{a}\}_{a\in Ag},\{R^{B}_{a}\}_{a\in Ag},v)$$
such that, for some $a,c\in Ag$:
\begin{itemize}
    \item $R^{I}_{a}:=\Delta_{W}\cup\{(i,w)\}$
    \item $R^{K}_{a}=R^{B}_{a}:=\Delta_{W}$;
    \item $R^{K}_{c}:=\Delta_{W}\cup\{(w,i)\}$;
    \item $R^{I}_{c}=R^{B}_{c}:=\{(i,i),(j,j),(w,i)\}$,
    \end{itemize}
    and, for any other $b\ne a,c$:
    \begin{itemize}
    \item $R^{K}_{b}=R^{B}_{b}=R^{I}_{b}:=\Delta_{W}\cup\{(i,j),(w,u)\}$.
\end{itemize}
Let $v$ be such that $v(q)=\{i,w,u\}$. Upon extending $v$ to an evaluation over the whole $\mathrm{Fm}_{\mathsf{S}}$, one has $\mathcal{M},i\models\sab(S_{c,b}q)$ but $\mathcal{M},i\not\models I_{a}\neg K_{b}q$. \qed
\end{enumerate}
\end{proof}

Let us briefly elaborate on Proposition \ref{prop:threeagents}. 
Items (1) and (2) directly result from our notion of a true secret. On the one hand, knowing that agent $a$ keeps $\varphi$ secret from agent $b$ entails knowing that $\varphi$ is true. On the other hand, if agent $a$ intends to keep $\varphi$ secret from agent $b$, then she intends to act in such a way that agent $b$ does not know that someone else is keeping $\varphi$ secret from her, on pain of disclosing $\varphi$ to agent $b$. 

Assuming by hypothesis that agent $a$ intends to keep $\varphi$ secret from agent $b$, the following items state conditions on trust from agent $a$ to another agent $c$. 
Item (3) establishes the two cases that convince agent $a$ that it is not worth keeping $\varphi$ secret from agent $c$ too: $(i)$ agent $a$ does not believe that agent $c$ communicates $\varphi$ to agent $b$, no matter the reason, or $(ii)$ agent $a$ does not intend to behave in such a way that agent $c$ can communicate $\varphi$ to agent $b$.
Item (4) states that if agent $a$ intends to communicate $\varphi$ to agent $c$, then she does not intend that agent $c$ may transfer $\varphi$ to agent $b$. Now, assuming that agent $a$ knows that agent $c$ keeps $\varphi$ secret from $b$ as well, items (5) and (6) establish that it is not always the case that agent $a$ intends to hide $c$'s secret to agent $b$. This depends on agent $a$'s intention of keeping $S_{c,b} \varphi$ true. In fact, agent $a$ may be motivated to convince agent $c$ that $\varphi$ is false in order to strengthen the secrecy of $\varphi$ to agent $b$.
Item (7) investigates the converse relation between the two secrets discussed above. Hiding from agent $b$ the fact that another agent $c$ is keeping $\varphi$ secret from agent $b$ does not mean we intend to hide $\varphi$ from agent $b$. For instance, an agent $a$ may want to reveal the secret $\varphi$ to agent $b$ without losing information about the secret keeper $c$. Note that items (1), (6), and (7) can also be obtained in \cite{Xiong2023}.

\section{Related works}\label{rel}

As we have already noted, there is a heterogeneous literature on the notion of secrecy, and logical papers such as~\cite{Ismail2020,Xiong2023} are also part of this literature. These works paved the way for the formal analysis of a commonsense notion of secret and the notion of knowing a secret, respectively.

Along the previous sections we have dealt with \cite{Xiong2023} in order to show that their framework can be somewhat expanded in order to extend the formal treatment of ``having a secret'' to ``intending to keep a true secret''. Indeed, we have shown that, once a simple intentionality operator is added, a large part of their theory can be preserved. We confine ourselves to observe that, once a wide framework accounting for both ``human'' and ``non-human'' dynamics of secrecy, and so eventually with non-perfect reasoners, is considered, then the Secret negation completeness' principle (cf. Proposition \ref{prop:notnegcompl}) need not hold, as there might be propositions that agents are not aware of. In turn, this entails that agents need not be always aware if they are willing to keep a piece of information secret or not. Since due references to the Xiong-{\AA}gotnes framework have been deepened along the preceding pages, we will not go any further, referring the interested reader to their work for details.

The system proposed in \cite{Ismail2020} is based on a sorted first-order language with equality, i.e., on a fragment of the VEL language of \cite{Bennett2004} equipped with a special sort for groups, two normal modal operators for belief and intention, and a non-normal modal operator for revelation. Thus, it is difficult to compare this system with ours from a syntactic and semantic point of view. However, some reflections on their definition of secrecy are in order. Indeed, in \cite{Ismail2020} Ismail and Shafie are after to formalize a notion of a secret involving a \emph{group} of secret keepers, a group of nescients, the content of a secret, a \emph{condition} under which the secret is supposed to be kept, and time. They provide the term-defined formula $\mathrm{Secret}_{0}(\varphi,K,N,\psi,t)$ whose intuitive reading is ``the group $K$ keeps $\varphi$ secret from the group $N$ under the condition $\psi$ at $t$'', and which subsumes the following facts. At time $t$, (i) it is true that $K$ will be possibly non-empty; (ii) any member $x$ of $K$ believes that $\varphi$ and $\psi$ and, moreover, that $N$ will be possibly non-empty; (iii) any member $x$ of $K$ does not believe that there is a member $y$ of $N$ whom $\varphi$ is revealed at $t$; (iv) any member $x$ of $K$ has the intention that in any time point $t'$ after $t$, as long as $\psi$ is true, $\varphi$ is revealed to no member of $N$.
 A key concept in the Ismail-Shafie approach is played by the notion of \textit{revelation}. As they point out (see \cite[p. 83]{Ismail2020}), they do not account \emph{acts} of revelation. Instead, they rely on a \emph{passive} notion of revelation, which allows us to interpret a statement like ``$\varphi$ is revealed to agent $x$'' as ``$\varphi$ is exposed or not covered to $x$''. As they remark, their notion of revelation is stronger than awareness but weaker than belief.
Upon considering their definition, it is not difficult to see that their notion of keeping a secret and the concept outlined in the present paper share a common flavor. In fact, in the light of their notion of revelation, understanding condition (iv) above as ``any secret keeper $a$ intends to bring about a state of affairs in which no nescient $b$ knows that $\varphi$'' is not a stretch. Therefore, their concept of keeping a secret is, at least in some regard, similar to ours, e.g., as it implies that secret keepers believe the secret is true, that it has not been revealed to nescients, and that secret keepers intend to preserve the ignorance of nescients about the (content of the) secret over time. However, it should be noted that, besides the obvious differences due to the employment of different syntactical and semantical settings, our framework and \cite{Ismail2020} differ in several aspects.
First, the system in \cite{Ismail2020} considers the belief operator $B$ as a $KD45$ operator and the intentionality $I$ as a $KD$ modal operator. This means that, concerning the belief operator, such a system is more ``powerful'' than ours. Still, in \cite{Ismail2020} intentionality is expressed as a $KD$ operator, while the operator $I$ we assume in this venue is $KD4$. Also, the interplay between belief and intentions in \cite{Ismail2020}  differs from ours. For example, axiom (A10) (and, of course, (A11)) are not  (\emph{mutatis mutandis}) mirrored in \cite{Ismail2020}. Conversely, we do not assume neither the negative introspection of agents with respect to intentions (axiom \textbf{IB1}) nor that agents which intend a given proposition $\varphi$ cannot believe that $\varphi$ is false (axiom \textbf{IB3}), see \cite[p. 84]{Ismail2020} for details. Finally, it should be remarked that in \cite{Ismail2020} true secrets and true ignorance of nescients are not considered as we do in the present paper. Therefore, as stated at the beginning of this section, our framework might be somehow regarded as ``intermediate'' between \cite{Xiong2023} and \cite{Ismail2020}. However,  it enjoys features that are not shared by previous attempts. 

\section{Conclusion}\label{concl}

In this work, we have introduced a robust formalization of the nuanced concept of intending to keep a true secret by pushing forward formal investigations carried out in \cite{Xiong2023} to capture the intricate dynamics of knowledge, beliefs, and intentions of secret keepers. The distinction between having a secret and intentionally keeping it is paramount, as it offers a richer perspective on the nature of secrecy. Indeed, the proposed formalization highlights three main components we believe should be taken into account in clarifying the concept of ``intending to keep $\varphi$ secret'': the secret keeper's knowledge that $\varphi$, her attribution of true ignorance to the agent/nescient she intends to keep $\varphi$  secret from, and her intention to act in such a way to ensure the persistence of the nescient's ignorance about $\varphi$. The last component is formalized by expanding the usual machinery of epistemic/doxastic logic by an intention operator, which enables a comprehensive understanding of agents' commitments to keep a secret.
Interestingly enough, the approach outlined in this paper suggests insightful considerations on ``keeping a true secret'' as well as it arguably provides a rather flexible framework which, 
positioned between existing approaches, adeptly navigates the static and intentional dimensions of secrecy and presents a versatile foundation opening avenues for diverse and impactful developments.

A first natural follow-up of our paper is investigating further metatheoretical properties of systems proposed in this venue, particularly regarding computational complexity issues with an eye to their applications.

Also, an important role in the dynamics of secrecy is played by the revelation of a true secret, which is meant as the act of communicating secret information with the expectation it still remains a secret to those not involved in the communication. Indeed, as it has been pointed out in the literature, e.g., in \cite{Bellman1981}, such a communicating act brings a metacommunicative dimension. Quoting B.L.~Bellman:
\begin{quote}
Secrecy is metacommunicative because when one hears the telling of a secret, several implicit instructions accompany it and constitute its key. That includes not only how the talk is to be understood but also that the information is not to be repeated and that the source where the knowledge was obtained is to be protected \cite[p. 9]{Bellman1981}.
\end{quote}
Therefore, a natural follow-up would be investigating, from a formal perspective, a suitable notion of revealing true secrets capable of encoding metacommunicative conditions it underpins. 
Along this line, it is also worth comparing 
our approach with the communication-based model of dynamic epistemic logic, with the aim of studying how intentions may evolve in response to communicative updates and how states accessible to $I_a$ may interact with communicative dynamics. Exploring the interaction between intentions and dynamic epistemic modes could provide insights into scenarios in which intentions and agents' epistemic states change simultaneously or influence each other, which could be shaped by a combination of these frameworks. 

As we have formalized it, the notion of ``intending to keep a true secret'' involves two agents only, i.e., a secret keeper and a nescient. However, real-world scenarios foresee many situations in which secret keepers, as well as nescients, form disjoint groups, each one with, e.g., specific inner ``epistemic dynamics''. For instance, each member $a$ of a group of secret keepers may know (or believe) any other secret keeper $c$ knows $a$ is among secret keepers. Generally, any group member might have reasonable expectations about the intentions, beliefs, and knowledge of other members of the same group. Therefore, we intend to further research the formal treatment of ``group secrets'' and how they change once group dynamics have been modified. We remark that a preliminary investigation of the above topics has been already carried out by the authors of the present paper in \cite{Aldini2025}.

As we delve into potential applications, the first important task is to make our proposed model susceptible to dynamic evolution.
In particular, a natural progression involves imbuing the model with temporal dynamics, enabling a more sophisticated representation of how secrets unfold over time. Integrating dynamic operators can capture the nuances of revealed secrets and the subsequent adjustments in agents' beliefs.
Moving beyond temporal aspects, we argue that our model can be enriched to unravel intricate relationships among agents. Exploring how an agent's knowledge and intentions influence the dynamics of secrecy, among others, adds a layer of complexity essential for understanding real-world social interactions.

Translating theoretical advancements into practical applications is a compelling direction, especially in cybersecurity. Addressing challenges such as safeguarding sensitive information \cite{10.1145/3289255}, managing access controls \cite{10.1007/978-3-540-78499-9_16,DBLP:conf/stairs/GenoveseRGT10}, and detecting internal threats within our framework holds promise for enhancing security protocols.
Moreover, representing and analyzing complex security policies within our framework is an intriguing avenue. This entails incorporating temporal constraints and accounting for contingent conditions to offer a comprehensive approach to security management.

Real-world applicability calls for experimental scrutiny to evaluate the model's effectiveness in capturing the intricacies of secrecy dynamics. Practical implementation in artificial intelligence or multi-agent systems, potentially leveraging machine learning, could bridge the gap between theoretical advancements and tangible impact.
As we envision the model's journey, multidimensional extensions come into play. 
To broaden the model's applicability, we aim to handle secrecy in contexts where agents have diverse sensitivity attributes, trust relationships~\cite{10.1093/logcom/exac016,10.1007/978-3-030-86772-0_41,10.1016/j.ijar.2024.109167,TagliaferriOslo}, or multiple objectives.
In conclusion, we argue that, due to its versatility, our model lends itself to a rich tapestry of possibilities. From dynamic refinements to real-world applications, it beckons researchers to chart new territories in the ever-evolving landscape of secrecy dynamics.
\subsection*{Acknowledgements}
The work of A. Aldini and P. Graziani was supported by the Italian Ministry of Education, University and Research through the PRIN 2022 project ``Developing Kleene Logics and their Applications'' (DeKLA), project
code: 2022SM4XC8. 
The work of D. Fazio and R. Mascella has been funded by the European Union - NextGenerationEU under the Italian Ministry of University and Research (MUR) National Innovation Ecosystem grant ECS00000041 - VITALITY - CUP C43C22000380007.

\bibliographystyle{plain}
\bibliography{bibliography.bib}

\appendix

\section{Some proofs}

\noindent In this section we outline the proof of  Theorem \ref{thm:dec}.\\ 

\subsection*{Proof of Theorem \ref{thm:dec}.} Since $\mathsf{S}$ is a finitely axiomatized normal modal system, it would suffice to show that it enjoys the finite model property (see, e.g. \cite[Theorem 8.15]{Hughes1}). To this aim, we put in good use the customary proof strategy of \emph{filtrations}.\\
By Theorem \ref{compl-sound}, for any $\varphi\in\mathrm{Fm}_{\mathsf{S}}$, if $\not\vds\varphi$, then there is an $\mathcal{S}$-model $\mathcal{M}$ falsifying it. If we are able to prove that, if such a model exists, then there is also a \emph{finite} $\mathcal{S}$-model $\mathcal{M}^{*}$ such that $\mathcal{M}^{*}\not\models\varphi$, one may provide a decision procedure for the validity of formulas as follows. We generate finite $\mathcal{S}$-models, as well as finite $\mathsf{S}$-proofs, in some order. If $\varphi$ is falsifiable, then the procedure will return a finite model falsifying it. On the contrary, the procedure will return a finite (since $\mathsf{S}$ is finitely axiomatized) proof for it (cf. \cite[Theorem 8.15]{Hughes1}).\\

Let $t:\mathrm{Fm}_{\mathsf{S}}\to\mathrm{Fm}_{\mathsf{S}}$ be defined recursively a follows:
\begin{itemize}
    \item[] $t(p):= p$, for any $p\in Var$;
    \item[] $t(\neg\varphi)=\neg t(\varphi)$;
    \item[] $t(\varphi\land\psi)=t(\varphi)\land t(\psi)$;
    \item[] $t(\star_{a}\varphi)=\star_{a}t(\varphi)$, for any $\star\in\{K,B\}$, and $a\in Ag$;
    \item[] $t(I_{a}\varphi)=I_{a}K_{a}t(\varphi)$ ($a\in Ag$).
\end{itemize}
The following lemma can be stated by means of a customary induction on the structure of formulas upon noticing that, for any $\varphi\in\mathrm{Fm}_{\mathsf{S}}$, $\vds I_{a}\varphi\leftrightarrow I_{a}K_{a}\varphi$ (left to the reader).

\begin{lemma}\label{lem:auxdec1}Let $\mathcal{M}=(W,\{R^{K}_{a}\}_{a\in Ag},\{R^{I}_{a}\}_{a\in Ag},\{R^{B}_{a}\}_{a\in Ag},v)$ be an $\mathcal{S}$-model. Then, for any  $\varphi\in\mathrm{Fm}_{\mathsf{S}}$, \[\mathcal{M}\models\varphi\text{ iff }\mathcal{M}\models t(\varphi).\]
\end{lemma}

For any $\varphi\in \mathrm{Fm}_{\mathsf{S}}$, let $\Theta_{\varphi}$ be the set of subformulas of $\varphi$ (cf. \cite[p. 136]{Hughes1}). Let $\varphi\in\mathrm{Fm}_{\mathsf{S}}$ and consider an $\mathcal{S}$-model $\mathcal{M}=(W,\{R^{I}_{a}\}_{a\in Ag},\{R^{K}_{a}\}_{a\in Ag},\{R^{B}_{a}\}_{a\in Ag},v)$. Let $\approx_{t(\varphi)}\subseteq W^{2}$ be defined as follows, for any $i,j\in W$:
\[i\approx_{t(\varphi)}j\text{ iff, for any }\psi\in \Theta_{t(\varphi)},\ \mathcal{M},i\models\psi\text{ iff }\mathcal{M},j\models\psi.\]
Clearly, $\approx_{t(\varphi)}$ is an equivalence relation over $W$. Let us consider the set $W^{*}$ obtained by picking exactly one representative from each equivalence class $[i]_{\approx_{t(\varphi)}}$ ($i\in W$). Note that, since $t(\varphi)$ is of finite length, $W^{*}$ is finite. Moreover, let $v^{*}:Var\to\mathcal{P}(W^{*})$ be defined as $v^{*}(p)=v(p)\cap W^{*}$. Let $\{\widetilde{R}^{K}_{a}\}_{a\in Ag}$, $\{\widetilde{R}^{I}_{a}\}_{a\in Ag}$, and $\{\widetilde{R}^{B}_{a}\}_{a\in Ag}$ be families of binary relations over $W^{*}$. The model (to be meant as usually) $$\mathcal{M}^{*}=(W^{*},\{\widetilde{R}^{K}_{a}\}_{a\in Ag},\{\widetilde{R}^{I}_{a}\}_{a\in Ag},\{\widetilde{R}^{B}_{a}\}_{a\in Ag},v^{*})$$ is a \emph{filtration} of $\mathcal{M}$ \emph{through} $t(\varphi)$ (cf. \cite[p. 138]{Hughes1}) provided that, for any $a\in Ag$, $\star\in\{K,I,B\}$:
\begin{itemize}
    \item For any $i,j\in W^{*}$, if there is some $k\in W$ such that $R^{\star}_{a}(i,k)$ and $k\approx_{t(\varphi)}j$, then $\widetilde{R}^{\star}_{a}(i,j)$;
    \item For any $i,j\in W^{*}$, if $\widetilde{R}^{\star}_{a}(i,j)$, then, for any $\star_{a}\psi\in\Theta_{t(\varphi)}$: $$\mathcal{M},i\models \star_{a}\psi\text{ implies }\mathcal{M},j\models\psi.$$
\end{itemize}

The next lemma can be proven as for \cite[Theorem 8.1]{Hughes1}.

\begin{lemma}\label{lem:auxdec2} Let $\mathcal{M}=(W,\{R^{K}_{a}\}_{a\in Ag},\{R^{I}_{a}\}_{a\in Ag},\{R^{B}_{a}\}_{a\in Ag},v)$ be an $\mathcal{S}$-model and $\varphi\in\mathrm{Fm}_{\mathsf{S}}$. If $\mathcal{M}^{*}=(W^{*},\{\widetilde{R}^{K}_{a}\}_{a\in Ag},\{\widetilde{R}^{I}_{a}\}_{a\in Ag},\{\widetilde{R}^{B}_{a}\}_{a\in Ag},v^{*})$ is a filtration of $\mathcal{M}$ (through $\approx_{t(\varphi)}$), then, for any $\psi\in\Theta_{t(\varphi)}$, $w\in W^{*}$, $$\mathcal{M}^{*},w\models\psi\text{ iff }\mathcal{M},w\models\psi.$$
\end{lemma}

Consequently, if we are able to prove that for any $\psi\in\mathrm{Fm_{\mathsf{S}}}$ and any $\mathcal{S}$-model $\mathcal{M}$, one can find a filtration $\mathcal{M}^{*}$ through $t(\psi)$ which is also an $\mathcal{S}$-model, then by Theorem \ref{compl-sound}, Lemma \ref{lem:auxdec1} and Lemma \ref{lem:auxdec2} we have that, for any $\varphi\in\mathrm{Fm}_{\mathsf{S}}$, if $\not\vds \varphi$, then $\mathcal{M}'\not\models\varphi$, for some finite model $\mathcal{M}'$.\\ 

\noindent Let us fix an arbitrary $\varphi\in\mathrm{Fm}_{\mathsf{S}}$ and an $\mathcal{S}$-model $$\mathcal{M}=(W,\{R^{I}_{a}\}_{a\in Ag},\{R^{K}_{a}\}_{a\in Ag},\{R^{B}_{a}\}_{a\in Ag},v).$$ Let $W^{*}$ and $v^{*}$ be built as above through $t(\varphi)$. We define the following binary relations over $W^{*}$, for any $a\in Ag$ and for any $i,j\in W^{*}$:
\begin{itemize}
    \item[-] $i\overline{R}^{K}_{a}j$ if and only if: 
  \begin{itemize}
      \item If $K_{a}\psi\in \Theta_{t(\varphi)}$, then $\mathcal{M},i\models K_{a}\psi$ implies $\mathcal{M},j\models K_{a}\psi$;
      \item If $I_{a}\psi\in \Theta_{t(\varphi)}$, then  $\mathcal{M},i\models I_{a}\psi$ implies $\mathcal{M},j\models I_{a}\psi$;
       \item If $B_{a}\psi\in \Theta_{t(\varphi)}$, then $\mathcal{M},i\models B_{a}\psi$ implies $\mathcal{M},j\models B_{a}\psi$.
  \end{itemize}
  \item[-] $i\overline{R}^{I}_{a}j$ if and only if, for any $I_{a}\psi\in\Theta_{t(\varphi)}$, $\mathcal{M},i\models I_{a}\psi$ implies $\mathcal{M},j\models I_{a}\psi$ and $\mathcal{M},j\models \psi$.
  \item[-] $i\overline{R}^{B}_{a}j$ if and only if:
  \begin{itemize}
      \item If $K_{a}\psi\in \Theta_{t(\varphi)}$, then $\mathcal{M},i\models K_{a}\psi$ implies $\mathcal{M},j\models K_{a}\psi$;
      \item If $I_{a}\psi\in \Theta_{t(\varphi)}$, then  $\mathcal{M},i\models I_{a}\psi$ implies $\mathcal{M},j\models I_{a}\psi$;
       \item If $B_{a}\psi\in \Theta_{t(\varphi)}$, then $\mathcal{M},i\models B_{a}\psi$ implies $\mathcal{M},j\models B_{a}\psi$ and $\mathcal{M},j\models \psi$. 
  \end{itemize}
\end{itemize}

\begin{lemma} $\mathcal{M}^{*}=(W^{*},\{\overline{R}^{K}_{a}\}_{a\in Ag},\{\overline{R}^{I}_{a}\}_{a\in Ag},\{\overline{R}^{B}_{a}\}_{a\in Ag},v^{*})$ is a filtration of $\mathcal{M}$ through $t(\varphi)$.
\end{lemma}
\begin{proof}\label{lem:filtration}Let us prove that, for any $a\in Ag$, any $\star\in\{K,I,B\}$, and $i,j\in W^{*}$, if $R^{\star}_{a}(i,k)$ for some  $k\in W$ and $j\approx_{t(\varphi)}k$, then $\overline{R}^{\star}_{a}(i,j)$. We confine ourselves to the case $\star=I$ leaving the remaining cases to the reader. Assume the hypotheses of the claim. Let $I_{a}\psi\in\Theta_{t(\varphi)}$. Therefore, $\mathcal{M},i\models I_{a}\psi$ implies $\mathcal{M},i\models I_{a}I_{a}\psi$ (by A11). So $\mathcal{M},k\models I_{a}\psi$ and $\mathcal{M},k\models \psi$. However, $I_{a}\psi,\psi\in\Theta_{t(\varphi)}$. We conclude $\mathcal{M},j\models I_{a}\psi$ and $\mathcal{M},j\models \psi$.\\
Showing that, for any $\star\in\{K,I,B\}$, $a\in Ag$, $i,j\in W^{*}$, $\star_{a}\psi\in\Theta_{t(\varphi)}$, if  $\overline{R}^{\star}_{a}(i,j)$, then $\mathcal{M},i\models \star_{a}\psi$ implies $\mathcal{M},j\models \psi$ is straightforward. For example,  if $K_{a}\psi\in\Theta_{t(\varphi)}$ and $\overline{R}^{K}_{a}(i,j)$, then $\mathcal{M},i\models K_{a}\psi$ implies $\mathcal{M},j\models K_{a}\psi$ and so, by (A3), $\mathcal{M},j\models \psi$. \qed
\end{proof}
Finally, we prove the following lemma.

\begin{lemma}\label{thm: filtrationsmod}$\mathcal{M}^{*}=(W^{*},\{\overline{R}^{K}_{a}\}_{a\in Ag},\{\overline{R}^{I}_{a}\}_{a\in Ag},\{\overline{R}^{B}_{a}\}_{a\in Ag},v^{*})$ is an $\mathcal{S}$-model.
\end{lemma}
\begin{proof} We show that $(W^{*},\{\overline{R}^{K}_{a}\}_{a\in Ag},\{\overline{R}^{I}_{a}\}_{a\in Ag},\{\overline{R}^{B}_{a}\}_{a\in Ag})$ is an $\mathcal{S}$-frame. 
To this aim we prove that conditions (2)-(8) from Definition \ref{def:s-mod} hold for any $a\in Ag$.\\
(2) That $\overline{R}^{B}_{a}$ is serial is ensured by the fact that ${R}^{B}_{a}$ is so and $\mathcal{M}^{*}$ is a filtration. In fact, let $i\in W^{*}$. Then there is $j\in W$ s.t. $R^{B}_{a}(i,j)$. Let $k\in W^{*}$ be such that $j\approx_{t(\varphi)}k$. By Lemma \ref{lem:filtration}, one has $\overline{R}^{B}_{a}(i,k)$.\\
(3) $\overline{R}^{I}_{a}$ is serial by reasoning as for (2). Transitivity can be proven as for $\mathsf{K4}$ (see \cite[p. 144]{Hughes1}).\\
(4) The reflexivity of $\overline{R}^{K}_{a}$ can be proven as for \cite[Theorem 8.4]{Hughes1}, while transitivity is apparent.\\
(5)-(7) are straightforward and so they are left to the reader.\\
Finally, as regards (8), let  $\overline{R}^{I}_{a}(i,j)$ and $\overline{R}^{K}_{a}(j,w)$. Let $I_{a}\psi\in\Theta_{t(\varphi)}$. Note that $I_{a}\psi = I_{a}K_{a}\psi'$, for some $\psi'\in\mathrm{Fm}_{\mathsf{S}}$. Let $\mathcal{M},i\models I_{a}\psi$, then, by hypothesis, $\mathcal{M},j\models I_{a}\psi$ and $\mathcal{M},j\models\psi=K_{a}\psi'$. Since $K_{a}\psi'\in\Theta_{t(\varphi)}$, one has $\mathcal{M},w\models I_{a}\psi$ and also $\mathcal{M},w\models K_{a}\psi'=\psi$. This concludes the proof of the theorem. \qed
\end{proof}
Note that the last part of the proof of Lemma \ref{thm: filtrationsmod} motivates why, in order to show that, if a formula $\varphi$ is falsifiable, then it is in a finite model, we have carried out the filtration procedure through $t(\varphi)$ rather than on $\varphi$, and then exploited Lemma \ref{lem:auxdec1}. 
Upon recollecting the above results and remarks, Theorem \ref{thm:dec} is proved.\qed

\end{document}